\documentclass[a4paper,reqno]{amsart}

\usepackage[utf8]{inputenc}
\usepackage{graphicx}
\usepackage{amsmath}
\usepackage{amsfonts}
\usepackage{amssymb}
\usepackage{amsthm}
\usepackage{cleveref}
\usepackage{mathtools}
\usepackage{microtype}
\usepackage{fullpage}
\usepackage[foot]{amsaddr}
\usepackage[usenames,dvipsnames]{xcolor}
\usepackage{mathrsfs}  
\usepackage{wrapfig}
\usepackage{enumitem}
\usepackage{dsfont}
\usepackage{graphicx}
\usepackage{caption}

\numberwithin{equation}{section}

\theoremstyle{plain}
\newtheorem{theorem}{Theorem}[section]
\newtheorem{lemma}[theorem]{Lemma}
\newtheorem{corollary}[theorem]{Corollary}
\newtheorem{proposition}[theorem]{Proposition}

\theoremstyle{definition}
\newtheorem{definition}[theorem]{Definition}
\newtheorem{remark}[theorem]{Remark}

\begin{document}

\title[Aperiodic order, spectral metrics and Jarn\'{\i}k sets]{A classification of aperiodic order via spectral metrics and Jarn\'{\i}k sets}

\author{M.\ Gr\"oger, M.\ Kesseb\"ohmer, A.\ Mosbach, T.\ Samuel and M.\ Steffens}

\address[M.\ Gr\"oger]{Faculty of Mathematics and Computer Science, Friedrich-Schiller University Jena, Ernst-Abbe-\newline
\hspace*{1.35em}Platz 2, 07743 Jena, Germany}
\address[M.\ Kesseb\"ohmer, A.\ Mosbach and M.\ Steffens]{FB 3 -- Mathematik und Informatik, Universit\"at Bremen\\
\hspace*{1.35em}Bibliothekstr. 1, 28359 Bremen, Germany}
\address[T.\ Samuel]{Mathematics Department, California Polytechnic State University, San Luis Obispo, CA, USA}

\dedicatory{Dedicated to the memory of Bernd O.\ Stratmann (1957-2015) -- A good friend, colleague and mentor.}

\subjclass[2010]{52C23, 68R15, 94A55, 47C15, 11K60, 37C45, 11J70, 46L87}

\maketitle

\vspace*{-3.25em}

\begin{abstract}
Given an $\alpha > 1$ and a $\theta$ with unbounded continued fraction entries, we characterise new relations between Sturmian subshifts with slope $\theta$ with respect to (i) an $\alpha$-H\"oder regularity condition of a spectral metric, (ii) level sets defined in terms of the Diophantine properties of $\theta$, and (iii) complexity notions which we call $\alpha$-repetitive, $\alpha$-repulsive and $\alpha$-finite; generalisations of the properties known as linearly repetitive, repulsive and power free, respectively.  We show that the level sets relate naturally to (exact) Jarn\'{\i}k sets and prove that their Hausdorff dimension is $2/(\alpha + 1)$.
\end{abstract}

\section{Introduction and outline}

\subsection{Introduction}

Links between regularity of spectral metrics built from noncommutative representations (spectral triples) and aperiodic behaviour of Sturmian subshifts, in the case that the continued fraction entries of the slope are bounded, were first observed in \cite{KS:2012}.  We show new relations between regularity properties of spectral metrics of Sturmian subshifts (where the continued fraction entries of the slopes are unbounded), fractal level sets (defined in terms of the Diophantine properties of $\theta$) and related complexity properties (which generalise and extend known notions of aperiodic behaviour) of Sturmian subshifts.  Here, the nontrivial and challenging task was to determine the exact Diophantine condition on $\theta$ and the optimal regularity on the spectral metrics, namely well-approximable of $\alpha$-type (\Cref{defn:*}) and sequential H\"older regularity (\Cref{defn:holderregular1,defn:holderregular2}) respectively.  In defining the fractal level sets the so-called Jarn{\'\i}k sets surface in a very natural way, and as such, our findings provide a nice application of this prominent class of number theoretic sets to aperiodic order.

The full shift over a finite alphabet $\mathscr{A}$ is the $\mathbb{N}$-action given by the left-shift $\sigma$ on the set of infinite $\mathscr{A}$-valued sequences.  A subshift $X$ is the restriction of this dynamical system to a closed \mbox{$\sigma$-invariant} subspace, see \Cref{sec:defAO}; of particular interest are minimal aperiodic subshifts, the prototypes being Sturmian subshifts.  Properties of such subshifts are encoded in the $C^{*}$-algebras $C(X) \rtimes_{\sigma} \mathbb{Z}$ and $C(X)$.

The central object in Connes' theory of noncommutative geometry is that of a spectral triple, for which one of the predominant motivations is to analyse geometric spaces, or dynamical systems, using operator algebras, particularly $C^{*}$-algebras.  This idea first appeared in the work of Gelfand and Na\u{\i}mark \cite{Gelfand+Neumark}, where it was shown that a $C^{*}$-algebra can be seen as a generalisation of the ring of complex-valued continuous functions on a locally compact space.  In \cite{NCDGCones,C:1994}, Connes formalised the notion of noncommutative geometry and, in doing so, showed that the tools of Riemannian geometry can be extended to non-Hausdorff spaces known as ``bad quotients'' and to spaces of a ``fractal'' nature.  In particular, Connes proposed the concept of a spectral triple $(\mathcal{A}, H, D)$. The $C^{*}$-algebra $\mathcal{A}$ acts faithfully on the Hilbert space $H$ together with an (essentially) self-adjoint operator $D$, called the Dirac operator, which has compact resolvent and bounded commutator with the elements of a dense sub-$*$-algebra of $\mathcal{A}$.  Additionally, Connes defined a pseudo-metric on the state space $\mathcal{S}(\mathcal{A})$ of $\mathcal{A}$  analogous to how the Monge-Kantorovitch metric is defined on the space of Borel probability measures supported on a given compact metric space.  

Subsequently, Rieffel \cite{Ri2} and Pavlovi\'c \cite{Pav} established conditions under which Connes' pseudo-metric is a metric and conditions under which the topology of Connes' pseudo-metric is equivalent to the weak-${*}$-topology, see \Cref{prop:Not_Metric} for a counter-part to the metric results of \cite{Pav,Ri2} in our setting.

While spectral triples for cross product algebras of the form $C(X) \rtimes_{\sigma} \mathbb{Z}$ seem difficult to set up, see for instance \cite{BMR:2010,C:1989,Ri3} and references therein, there has been a lot of activity in constructing spectral triples for commutative $C^{*}$-algebras $C(Y)$, where $Y$ does not carry an obvious differential structure.  A series of works has been devoted to general metric spaces \cite{expfast-finte-spec1,Pav,Ri2,Ri3} and specially to sets of a fractal nature \cite{STCS,DOST,Bellisard,C:1994,Self-reference-1,GI1,GI2,JP:2015,KS:2012,KessSam:2013,Lapidus3,L:2008,Sharp:2012}. 

Kellendonk and Savinien \cite{KS:2012} proposed a modification of the spectral triple and spectral metric pioneered by Bellissard and Pearson \cite{Bellisard}, which in turn stems from the work of Connes \cite{C:1994} and Guido and Isola \cite{GI1,GI2}, that can be used to analysis Sturmians subshifts; this construction was later generalised to minimal subshifts over a finite alphabet in \cite{KLS:2011}.  It is with the spectral triple and spectral metric of \cite{KS:2012} that we will work.  As is often the case, once one is lead to consider certain objects by an abstract theory (here spectral triples) and these objects turn out to be useful in another field (here aperiodic order) one finds that they can be defined \emph{ad hoc}, that is, without any knowledge of the abstract theory.  This is the case here, and so, in the sequel we work with the combinatorial version of the spectral metric given in \cite{KS:2012}.

The essential ingredients in the construction of \cite{KS:2012} are an infinite weighted graph (augmented weighted tree, as introduced by Kaimanovich \cite{K:2001}), whose (hyperbolic) boundary is homeomorphic to the given Sturmian subshift, and the notion of a choice function, which can be seen as the noncommutative analogue of a vector in the tangent space of a Riemannian manifold.  The main result of \cite{KS:2012} showed that the spectral metric is Lipschitz equivalent to the underlying ultra metric if and only if the continued fraction entries of the slope of the Sturmian subshift are bounded, which in turn is equivalent to several known notions of aperiodic behaviour, as we will shortly explain in further detail.  The typical choice of weightings used in \cite{KS:2012} to define the spectral triple (in particular the Dirac operator) is suggested to be $\delta_{n} = \ln(n) n^{-t}$, where $t > 0$, and investigations of spectral metrics when $\delta_{n} = n^{-1}$ have recently been carried out in \cite{JP:2015}.  Thus our choice of the weightings $\delta_{n} = n^{-t}$, for $t > 0$, is a natural choice, generalising and extending this line of research.

In the case of a Sturmian subshift of slope $\theta$ having unbounded continued fraction entries, in \Cref{thm:thm2,thm:thm2.5} and \Cref{cor:cor-main}, we give necessary and sufficient conditions  (well-approximable of $\alpha$-type, see \Cref{defn:*}) on the Diophantine properties of $\theta$ for when the spectral metric, proposed in \cite{KS:2012}, is sequentially H\"older regular to the underlying ultra metric.  Moreover, we show that the sequential H\"older regularity cannot be strengthen to H\"older equivalence, see \Cref{rmk:strengthening}.  Additionally, in \Cref{thm:thm3} we compute the Hausdorff dimension of the set $\Theta_{\alpha}$ of $\theta$'s which are well-approximable of $\alpha$-type, by relating $\Theta_{\alpha}$ to Jarn\'{\i}k and exact Jarn\'{\i}k sets and by using the results of \cite{B:2003,B:2008,DK:2002,K:1997}.

The theory of aperiodic order is a relatively young field of mathematics which has attracted considerable attention in recent years, see for instance \cite{AR:1991,BG:2013,BaMo:2000,BFMS:2002,FGJ:2015,KLS:2011,KS:2012,Mo:1997,MH:1940,P:1998,S:2015}.  It has grown rapidly over the past three decades; on the one hand, due to the experimental discovery of physical substances, called quasicrystals, exhibiting such features \cite{INF:1985,SBGC:1984}; and on the other hand, due to intrinsic mathematical interest in describing the very border between crystallinity and aperiodicity.  Here, of particular interest are point sets, such as Delone sets, of which Sturmian subshifts are the quintessential examples.  While there is no axiomatic framework for aperiodic order, various types of order conditions, in terms of complexity, have been studied, see \cite{AR:1991,BG:2013,BFMS:2002,FGJ:2015,HKW:2015,KS:2012,Lag:1999,Lag:2002,Lag:2003,MH:1940} and references therein.  Such order conditions include linear repetitiveness, repulsiveness and power freeness.  Here, we introduce generalisations and extensions of these notions (\Cref{defn:repetitive,defn:repulsive,defn:free}) and show the exact impact these new notions have on the Diophantine properties of $\theta$, see \Cref{thm:thm1}.  This generalises and extends the well-known result \cite{BFMS:2002,D:2000,KS:2012,MH:1940} that the following are equivalent.
\begin{enumerate}[leftmargin=2.5em]
\item\label{item:orginal_ch_1} A Sturmian subshift is linearly repetitive.
\item\label{item:orginal_ch_2} A Sturmian subshift is repulsive.
\item\label{item:orginal_ch_3} A Sturmian subshift  is power free.
\item\label{item:orginal_ch_4} The continued fraction entries of $\theta$ are bounded.
\end{enumerate}
Such notions of complexity, and the associated implications on the Diophantine properties of $\theta$, correspond to properties of the dynamical system and hence of the $C^{*}$-algebras $C(X) \rtimes_{\sigma} \mathbb{Z}$ and $C(X)$.  Therefore, it is natural to consider spectral triples with these algebras and to compare how these can be used to classify  Sturmian subshifts in terms of  the Diophantine properties of $\theta$.  Indeed, we show precisely how our new notions are related to each other (\Cref{thm:thm1}) and to the sequentially H\"older regularity of the spectral metric (\Cref{thm:thm2,thm:thm2.5} and \Cref{cor:cor-main}).  Note, in \cite{KLS:2011,S:2015} the equivalence of \eqref{item:orginal_ch_2} and the Lipschitz equivalence of the ultra and spectral metric was generalised to minimal aperiodic subshifts over a finite alphabet and tilings.

Extending our results concerning the sequential H\"older regularity of the spectral metric and the new complexity concepts we introduce (\Cref{defn:repetitive,defn:repulsive,defn:free}) to suitable $S$-adic subshifts would be a worthwhile and fruitful venture.  Such suitable $S$-adic subshifts should consist of those which allow to describe their letter frequencies by a multidimensional continued fraction algorithm, for example Arnoux-Rauzy, Brun and Jacobi-Perron subshifts, see \cite{B:2014,B:2016} and references therein for further details.  That such an extension is feasible can in principle be seen in the work of \cite{KLS:2011}.  Further, we remark that a class of $S$-adic subshifts, referred to as (generalised) Grigorchuk subshifts, have been investigated in this context in a sequel to this article \cite{DKMSS:2017}.  The construction of these subshifts were inspired by Lysenok's \cite{L:85} presentation of the Grigorchuk group $G$ (the first known finitely generated group to exhibit intermediate growth) and have been shown to exhibit a rich variety of behaviours, see for instance \cite{DKMSS:2017,GLN:16,GLN:16b}.

\subsection{Outline}

In the following section, we present all of the necessary notations and definitions required to state our main results.  This section is broken down into three parts; definitions concerning continued fraction expansions (\Cref{sec:defCF}), definitions concerning Sturmain subshifts and aperiodic order (\Cref{sec:defAO}) and definitions concerning spectral metrics (\Cref{sec:defnSM}).  In \Cref{sec:main}, we present our main results; \Cref{thm:thm2,thm:thm2.5,thm:thm1,thm:thm3} and \Cref{cor:cor-main}.  We then give several preliminaries on Sturmian subshifts (\Cref{sec:Prelim_AO}) and spectral metrics (\Cref{sec:Pre_lim_Spectral_Metric}).  After which in \Cref{sec:1-repulsive=repulsive} we present the proofs of \Cref{prop:1-repulsive=repulsive,prop:Not_Metric,prop:liminf_psi_r}; \Cref{prop:1-repulsive=repulsive} demonstrate why our new notions of complexity are generalisations and extensions of existing forms of complexity; \Cref{prop:Not_Metric} gives a condition when the spectral metric is not a metric; and \Cref{prop:liminf_psi_r} justifies the limit superior in the definition of sequential H\"older regularity.  In \Cref{sec:proof_thm2+.5}, the proofs of \Cref{thm:thm2,thm:thm2.5} are given, in \Cref{sec:Proof_Thm_1} we present the proof of \Cref{thm:thm1} and we conclude with the proof of \Cref{thm:thm3} on the Hausdorff dimension of the set $\Theta_{\alpha}$ in \Cref{sec:proof_thm3}.

\section{Notation and Definitions}\label{sec:def}

\subsection{Continued fractions}\label{sec:defCF}

Here, we review the definition of continued fraction expansions and introduce the new concept of well-approximable of $\alpha$-type.

Let $\theta \in [0, 1]$ denote an irrational number.  For a natural number $n \geq 1$, set $a_{n} = a_{n}(\theta) \in \mathbb{N}$ to be the $n$-th continued fraction entry of $\theta$,  that is
\begin{align*}
\theta = [0; a_{1} , a_{2}, \dots ] \coloneqq  \cfrac{1}{a_{1} + \cfrac{1}{ a_{2} + \cdots}\;\cdot}
\end{align*}
We let 
$q_{0} = q_{0}(\theta) \coloneqq 1$,
$q_{1} = q_{1}(\theta) \coloneqq a_{1}$,
$p_{0} = p_{0}(\theta) \coloneqq 0$, and
$p_{1} = p_{1}(\theta) \coloneqq 1$
and for a given integer $n \geq 2$, we set
\begin{align*}
q_{n} = q_{n}(\theta) \coloneqq a_{n} q_{n-1} + q_{n-2}
\quad \text{and} \quad
p_{n} = p_{n}(\theta) \coloneqq a_{n} p_{n-1} + p_{n-2}.
\end{align*}
It is known that $\text{gcd}(p_{n}, q_{n}) = 1$ and that $p_{n}/q_{n} = [0; a_{1}, \dots, a_{n} ]$, for all $n \in \mathbb{N}$, see for instance \cite{DK:2002,K:1997}.  We observe that if $a_{1} = 1$, then
\begin{align}\label{eq:1-theta1}
\theta =  [0; 1, a_{2}, a_{3}, \dots ] > 1/2,
\quad
1 - \theta = [0; a_{2} + 1, a_{3}, \dots ]
\quad \text{and} \quad
q_{n}(\theta) = q_{n-1}(1-\theta).
\end{align}

\begin{definition}\label{defn:*}
For $\alpha \geq 1$ and an irrational number $\theta \in [0,1]$, set $\displaystyle \textup{A}_{\alpha}(\theta) \coloneqq \limsup_{n \to \infty} a_{n} q_{n-1}^{1-\alpha}$ and define
\begin{align*}
\underline{\Theta}_{\alpha} \coloneqq \{ \theta \in [0, 1] \colon 0 < A_{\alpha}(\theta) \},
\quad
\overline{\Theta}_{\alpha} \coloneqq \{ \theta \in [0, 1] \colon A_{\alpha}(\theta) < \infty \}
\quad \text{and} \quad
\Theta_{\alpha} \coloneqq \underline{\Theta}_{\alpha} \cap \overline{\Theta}_{\alpha}.
\end{align*}
Further, we say that $\theta$ is
\begin{enumerate}[leftmargin=2.5em]
\item \emph{well-approximable of $\overline{\alpha}$-type}, if $\textup{A}_{\alpha}(\theta) < \infty$,
\item \emph{well-approximable of $\underline{\alpha}$-type}, if $\textup{A}_{\alpha}(\theta) > 0$,
\item \emph{well-approximable of $\alpha$-type}, if $0 < \textup{A}_{\alpha}(\theta) < \infty$.
\end{enumerate}
\end{definition}

Notice, any irrational $\theta \in [0, 1]$ is well-approximable of $\underline{1}$-type. Further, the condition that an irrational $\theta \in [0, 1]$ is well-approximable of $\overline{1}$-type, and hence of ${1}$-type, is equivalent to the continued fraction entries of $\theta$ being bounded.

\begin{proposition}\label{prop:theta-1-theta}
For an irrational $\theta$, we have that
\begin{enumerate}[leftmargin=2.5em]
\item $\theta$ is well-approximable of $\overline{\alpha}$-type if and only if $1 - \theta$ is well-approximable of $\overline{\alpha}$-type, and
\item $\theta$ is well-approximable of $\underline{\alpha}$-type if and only if $1 - \theta$ is well-approximable of $\underline{\alpha}$-type.
\end{enumerate}
\end{proposition}

\begin{proof}
This is a consequence of \eqref{eq:1-theta1} and \Cref{defn:*}.
\end{proof}

\subsection{Sturmian subshifts and aperiodic order}\label{sec:defAO}

Here, we review the key definitions of subshifts and introduce three new concepts of complexity: $\alpha$-repetitive, $\alpha$-repulsive and $\alpha$-finite, for a given $\alpha \geq 1$.

For $n \in \mathbb{N}$ we define $\{ 0, 1 \}^{n}$ to be the set of all finite words in the alphabet $\{ 0, 1 \}$ of length $n$, and set
\begin{align*}
\{ 0, 1 \}^{*} \coloneqq \bigcup_{n \in \mathbb{N}_{0}} \{ 0, 1 \}^{n},
\end{align*}
where by convention $\{ 0, 1 \}^{0}$ is the set containing only the \emph{empty word} $\emptyset$.  We denote by $\{ 0, 1 \}^{\mathbb{N}}$ the set of all infinite words and equip it with the discrete product topology.
The continuous map $\sigma \colon \{ 0, 1 \}^{\mathbb{N}} \to \{ 0, 1 \}^{\mathbb{N}}$ defined by $\sigma( x_{1}, x_{2}, \dots ) \coloneqq ( x_{2}, x_{3}, \dots )$ is called the \emph{left-shift}.  A closed set $Y \subseteq\{ 0, 1 \}^{\mathbb{N}}$ which is left-shift invariant (that is $\sigma(Y) = Y$) is called a \emph{subshift}.  On every subshift $Y$ we can define a metric inducing the product topology: let $\delta = (\delta_{n})_{n \in \mathbb{N}}$ be a strictly decreasing null sequence of positive real numbers and define $d_{\delta} \colon Y \times Y\to \mathbb{R}$ via $d_{\delta}(v, w) \coloneqq \delta_{\lvert v \vee w \rvert}$.   Here, $\lvert v \vee w \rvert$ denotes the length of the longest prefix which $v$ and $w$ have in common, and if there is no such prefix, then we set $\lvert v \vee w \rvert = 1$.  (Recall that a finite word $u$ is called a \textit{prefix} of a finite or infinite word $v$ if there exists a word $u'$ such that $v = u u'$.  Similarly, a finite word $u$ is called a \textit{suffix} of a finite word $v$ if there exists a finite word $v'$ such that $v = v' u$.)

There are plenty of ways to introduce Sturmian subshifts. For example, they can be defined via a cut and project scheme \cite{BG:2013}, as extensions of circle rotations \cite{BFMS:2002}, using a substitution sequence \cite{MH:1940} or, as in the definition below, via so-called mechanical (infinite) words, also known as rotation sequences, see for instance \cite{BFMS:2002,Lot:2002}.

\begin{definition}\label{defn:sturmain_subshift} 
Let $\theta \in [0,1]$ be an irrational number and define the \emph{rotation sequence $x \coloneqq (x_{n})_{n \in \mathbb{N}}$ for $\theta$} by $x_{n} \coloneqq \lceil \theta(n+1) \rceil - \lceil \theta n \rceil$.  The set $\Omega(x) \coloneqq \overline{\{ \sigma^{k}(x_{1}, x_{2}, \dots) \colon k \in \mathbb{N}_{0} \}}$ is called the \emph{Sturmian subshift} of slope $\theta$.
\end{definition}

\begin{theorem}[\cite{BFMS:2002}]
A Sturmian subshift is aperiodic and minimal with respect to $\sigma$.
\end{theorem}

For $w = (w_{1}, w_{2}, \dots, w_{k})$ and $v = (v_{1}, v_{2}, \dots, v_{n}) \in \{ 0, 1\}^{*}$, we set $w v \coloneqq (w_{1}, w_{2}, \dots, w_{k}, v_{1}, v_{2}, \dots, v_{n})$, that is the \emph{concatenation} of $w$ and $v$.  Note that  $\{ 0, 1\} ^{*}$ together with the operation of concatenation defines a semigroup. The \emph{length} of $v$ is the integer $n$ and is denoted by $\lvert v \rvert$.  We set $v\lvert_{m} \coloneqq (v_{1}, v_{2}, \dots, v_{m})$ for all integers $0 \leq m \leq n = \lvert v \rvert$.  Further, we say that a word $u \in  \{ 0, 1\}^{*}$ is a \emph{factor} of $v$ if there exists an integer $j$ with $u = \sigma^{j}(v)\lvert_{\lvert u \rvert}$.  We use the same notations when $v$ is an infinite word.  The \emph{language} $\mathcal{L}(Y)$ of a subshift $Y$ is the set of all factors of the elements of $Y$.  Following convention, the empty word $\emptyset$ is assumed to be contained in the language $\mathcal{L}(Y)$.  We call $w \in \mathcal{L}(Y)$ a \emph{right special word} if both $w (0)$ and $w (1)$ belong to $\mathcal{L}(Y)$. We denote the set of right special words by $\mathcal{L}_{\textup{R}}(Y)$; following convention, we assume $\emptyset \in \mathcal{L}_{\textup{R}}(Y)$.

\begin{remark}\label{rmk:eta}
Let $\eta$ denote the involution on $ \{ 0, 1\}^{\mathbb{N}}$ given by  $\eta(w_{1}, w_{2}, \dots ) \coloneqq (e(w_{1}), e(w_{2}), \dots )$ with $e(0) \coloneqq 1$ and $e(1) \coloneqq 0$.  Let $x$ be the rotation sequence for $\theta$ and $y$ be the rotation sequence for $1-\theta$, then by a result of \cite{BFMS:2002}, it follows that $\Omega(x) = \Omega(\eta(y))$.
\end{remark}

\begin{remark}\label{rmk:BFMS-2002}
A known characterisation of (Epi-) Sturmian subshifts (over a finite alphabet) is that $\mathcal{L}(X)$ contains a unique right special word per length, see for instance \cite{BFMS:2002}.  Further, if $w \in \mathcal{L}_{\textup{R}}(X)$, then $\sigma^{k}(w)$ is a right special word for all $k \in \{ 1, 2, \dots, \lvert w \rvert \}$.
\end{remark}

\begin{definition}\label{defn:rep_fun}
The \emph{repetitive function} $R \colon \mathbb{N} \to \mathbb{N}$ of a subshift $Y$ assigns to $r$ the smallest $r'$ such that any element of $\mathcal{L}(Y)$ with length $r'$ contains (as factors) all elements of $\mathcal{L}(Y)$ with length $r$.
\end{definition}

\begin{theorem}[\cite{MH:1940}]\label{thm:MH-1940}
For $X$ a Sturmian subshift of slope $\theta \in [0,1]$, we have that
\begin{align*}
R(n) = \begin{cases}
R(n - 1) + 1 & \text{if} \; n \in \mathbb{N} \setminus \{ q_{k} \}_{k \in \mathbb{N}},\\
q_{k+1} + 2 q_{k} -1 & \text{if} \; n = q_{k} \; \text{for some} \; k \in \mathbb{N}.
\end{cases}
\end{align*}
\end{theorem}

\begin{definition}\label{defn:repetitive}
Let $\alpha \geq 1$ be given and set
\begin{align*}
R_{\alpha} \coloneqq \limsup_{n \to \infty} \frac{R(n)}{n^{\alpha}}.
\end{align*}
A subshift $Y$ is called \emph{$\alpha$-repetitive} if $R_{\alpha}$ is finite and non-zero.
\end{definition}

The notion of $1$-repetitive implies linearly recurrent and in the case of a Sturmian subshift these notions coincide.  Also, if $1 \leq \alpha < \beta$ and $0 < R_{\beta} < \infty$, then $R_{\alpha} = \infty$.  Similarly, if $0 < R_{\alpha} < \infty$, then $R_{\beta} = 0$.

After the completion of this article the authors learnt that the term $\alpha$-repetitive has been used before, see for instance \cite{LNCS4649}.  However, the definition given above and that given in \cite{LNCS4649}, although related, record different information.  We refrain from giving the precise definition of \cite{LNCS4649} here as we believe it would not provide further inside to the reader.

\begin{definition}\label{defn:repulsive}
For $\alpha \geq 1$ and for a subshift $Y$, set $\displaystyle \ell_{\alpha} \coloneqq \liminf_{n \to \infty} A_{\alpha, n}$ where for $n \geq 2$ an integer
\begin{align*}
A_{\alpha, n} \coloneqq \inf \left\{ \frac{\lvert W \rvert - \lvert w \rvert}{\lvert w \rvert^{1/\alpha}} \colon w, W \in \mathcal{L}(Y), w \; \text{is a prefix and suffix of} \; W, \; \lvert W \rvert = n \; \text{and} \; W \neq w \neq \emptyset \right\}.
\end{align*}
If $\ell_{\alpha}$ is finite and non-zero, then we say that $Y$ is \emph{$\alpha$-repulsive}.
\end{definition}

We recall that a subshift $Y$ is called \emph{repulsive} if the value
\begin{align*}
\ell \coloneqq \inf \left\{ \frac{\lvert W \rvert - \lvert w \rvert}{\lvert w \rvert} \colon w, W \in \mathcal{L}(Y), w \; \text{is a prefix and suffix of} \; W, \; \text{and} \; W \neq w \neq \emptyset \right\}
\end{align*}
is non-zero.  The following proposition, which is proven in \Cref{sec:1-repulsive=repulsive}, relates the notions $1$-repulsive and repulsive.  In fact, in a sequel to this article \cite{DKMSS:2017}, this result is shown to hold for general subshifts over finite alphabets.  In particular, it was shown in \cite{KLS:2011} that power free and repulsive are equivalent  and, as we will shortly see, we have that power free and $1$-finite (\Cref{defn:free}) are equivalent.  The general result follows, by combining these observations with Theorem 3.1 of \cite{DKMSS:2017} where it is shown that $1$-repulsive and $1$-finite are equivalent.

\begin{proposition}\label{prop:1-repulsive=repulsive}
A Sturmian subshift is $1$-repulsive if and only if it is repulsive.
\end{proposition}

\begin{remark}\label{rmk:alpha-beta-repulsive}
If $1 \leq \alpha < \beta$ and if $0 < \ell_{\beta} < \infty$, then $\ell_{\alpha} = 0$.  To see this, suppose that $0 < \ell_{\beta} < \infty$.  Thus, for $n \in \mathbb{N}$ sufficiently large, there exist words $w, W \in \mathcal{L}(Y)$ with $w$ a prefix and suffix of $W$, $\lvert W \rvert = n$ and $W \neq w \neq \emptyset$, so that 
\begin{align*}
\frac{\ell_{\beta}}{2} \leq \frac{\lvert W \rvert - \lvert w \rvert}{\lvert w \rvert^{1/\beta}} \leq 2 \ell_{\beta}.
\end{align*}
Hence, $\lvert w \rvert \geq n (2 \ell_{\beta} + 1)^{-1}$, and
\begin{align*}
\frac{\ell_{\beta}  \lvert w \rvert^{1/\beta-1/\alpha}}{2} \leq \frac{\lvert W \rvert - \lvert w \rvert}{\lvert w \rvert^{1/\alpha}} \leq 2 \ell_{\beta} \lvert w \rvert^{1/\beta-1/\alpha}.
\end{align*}
Therefore, we have that $\ell_{\alpha} = 0$.
\end{remark}

The next definition is a generalisation of the notion of a subshift being power free.  Indeed, one sees that if $\alpha = 1$, then $1$-finite is equivalent to the (asymptotic) index being finite, which in turn is equivalent to the property of being power free. For further details on the index of Sturmian subshifts, see for instance \cite{A:2005,DL:2002}.

\begin{definition}\label{defn:free}
For a subshift $Y$ we define $Q \colon \mathbb{N} \to \mathbb{N} \cup \{ + \infty \}$ by
\begin{align*}
Q(n) \coloneqq \sup \{ p \in \mathbb{N} \colon \text{there exists} \; W \in \mathcal{L}(Y) \; \text{with} \; \lvert W \rvert = n \; \text{and} \; W^{p} \in \mathcal{L}(Y) \}.
\end{align*}
Let $\alpha \geq 1$ be given.  We say that the subshift $Y$ is \emph{$\alpha$-finite} if the value
\begin{align*}
Q_{\alpha} \coloneqq \limsup_{n \to \infty} \frac{Q(n)}{n^{\alpha - 1}}
\end{align*}
is non-zero and finite.
\end{definition}

\subsection{Spectral metric}\label{sec:defnSM}

Here, we give the definition of a spectral metric as introduced in \cite{KS:2012}; we also define sequential H\"older regularity of metrics.

As is often the case, once one is lead to consider certain objects by an abstract theory (here spectral triples) and these objects turn out to be useful in another field (here aperiodic order) one finds out that they can also be defined \emph{ad hoc}, that is, without any knowledge of the abstract theory.  This is the case here and so we present a combinatorial version of the spectral metric as introduced in \cite{KS:2012} and refer to \cite{KLS:2011,KS:2012} for the definition of the spectral triple used to defined the spectral metric.

Let $X$ denote a Sturmian subshift and let $\delta = (\delta_{n})_{n \in \mathbb{N}}$ denote a strictly decreasing null sequence of positive real numbers.  The spectral metric $d_{s, \delta} \colon X \times X \to \mathbb{R}$ is defined by
\begin{align}\label{eq:spectral_metric_bound}
d_{s, \delta}(v, w) \coloneqq \delta_{\lvert v \vee w \rvert} + \sum_{n > \lvert v \vee w \rvert} \overline{b}_{n}(v)  \delta_{n} + \sum_{n > \lvert v \vee w \rvert} \overline{b}_{n}(w) \delta_{n},
\end{align}
for all $v, w \in X$.  Here, for $n \in \mathbb{N}$ and $z = (z_{1}, z_{2}, \dots ) \in X$, we set 
\begin{align*}
\overline{b}_{n}(z) \coloneqq 
\begin{cases}
1 & \text{if} \; (z_{1}, z_{2}, \dots, z_{n}) \; \text{is a right special word},\\
0 & \text{otherwise}.
\end{cases}
\end{align*}
Setting $\delta_{n} = n^{-t}$, for $n \in \mathbb{N}$, the following result gives a necessary condition for when the spectral metric $d_{s, \delta}$ is not bounded; complementing \cite[Theorem 4.14]{KS:2012}.  The proof is presented in \Cref{sec:1-repulsive=repulsive}.

\begin{proposition}\label{prop:Not_Metric}
Let $\alpha > 1$ and let $X$ be a Sturmian subshift of slope $\theta \in \underline{\Theta}_{\alpha}$.  For $t \in (0, 1 - 1/\alpha]$, setting $\delta_{n} \coloneqq n^{-t}$, for $n \in \mathbb{N}$, the spectral metric $d_{s, \delta}$ is not a metric and for $t > 1 - 1/\alpha$, the spectral metric $d_{s, \delta}$ is a metric.
\end{proposition}

To define sequentially H\"older regularity we set for $w\in X$ and $r > 0$,
 \begin{align*}
\psi_{w}(r) \coloneqq \limsup_{v \, \xrightarrow[d_\delta]{} \, w} \frac{d_{s, \delta}(w, v)}{d_{\delta}(w, v)^{r}} \qquad \mbox{ and } \qquad \psi(r) \coloneqq \sup \{ \psi_{w}(r) \colon w \in X \}.
\end{align*} 
For all $r \in (0, 1)$, we observe that by replacing limit superior with limit inferior in the definition of $\psi_{w}(r)$, then $\psi(r) = 0$, compare with \Cref{thm:thm2,thm:thm2.5}.  The proof is presented in \Cref{sec:1-repulsive=repulsive}.

\begin{proposition}\label{prop:liminf_psi_r}
For $\alpha>1$ and $r \in (0, 1)$, we have that $\displaystyle{\liminf_{v \, \xrightarrow[d_\delta]{} \, w} \frac{d_{s, \delta}(w, v)}{d_{\delta}(w, v)^{r}} =0}$.
\end{proposition}

\vspace{-0.75em}

\begin{definition}\label{defn:holderregular1}
Let $r, \epsilon > 0$ be given.
\begin{enumerate}[leftmargin=2.5em]
\item The metric $d_{s, \delta}$ is \emph{sequentially $\overline{r}$-H\"older regular} to $d_{\delta}$ if $\psi(r) < \infty$.
\item The metric $d_{s, \delta}$ is \emph{sequentially $\underline{r}$-H\"older regular} to $d_{\delta}$ if $\psi(r) > 0$.
\item The metric $d_{s, \delta}$ is \emph{sequentially $r$-H\"older regular} to $d_{\delta}$ if $0 < \psi(r) < \infty$.
\end{enumerate}
\end{definition}

\vspace{-0.25em}

We will also require the following weaker notion of sequential H\"older regularity.

\vspace{-0.25em}

\begin{definition}\label{defn:holderregular2}
Let $r, \epsilon > 0$ be given.
\begin{enumerate}[leftmargin=2.5em]
\item The metric $d_{s, \delta}$ is \emph{critically sequentially $\overline{r}$-H\"older regular} to $d_{\delta}$ if $\psi(r - \epsilon) = 0$, for all $0<\epsilon < r$.
\item The metric $d_{s, \delta}$ is \emph{critically sequentially $\underline{r}$-H\"older regular} to $d_{\delta}$ if $\psi(r + \epsilon) = \infty$, for all $\epsilon > 0$.
\item The metric $d_{s, \delta}$ is \emph{critically sequentially $r$-H\"older regular} to $d_{\delta}$ if $d_{s, \delta}$ is critically sequentially $\overline{r}$- and $\underline{r}$-H\"older regular to $d_{\delta}$.
\end{enumerate}
\end{definition}

For a given $r \in (0, 1]$, if the metric $d_{s, \delta}$ is sequentially $\overline{r}$-H\"older (respectively, $\underline{r}$-H\"older) regular to $d_{\delta}$, then $d_{s, \delta}$ is critically sequentially $\overline{r}$-H\"older (respectively, $\underline{r}$-H\"older) regular to $d_{\delta}$.

\section{Main results}\label{sec:main}

\begin{wrapfigure}[7]{r}{0.3\textwidth}
\begin{center}
\vspace{-1.5em}
\includegraphics[width=0.23\textwidth]{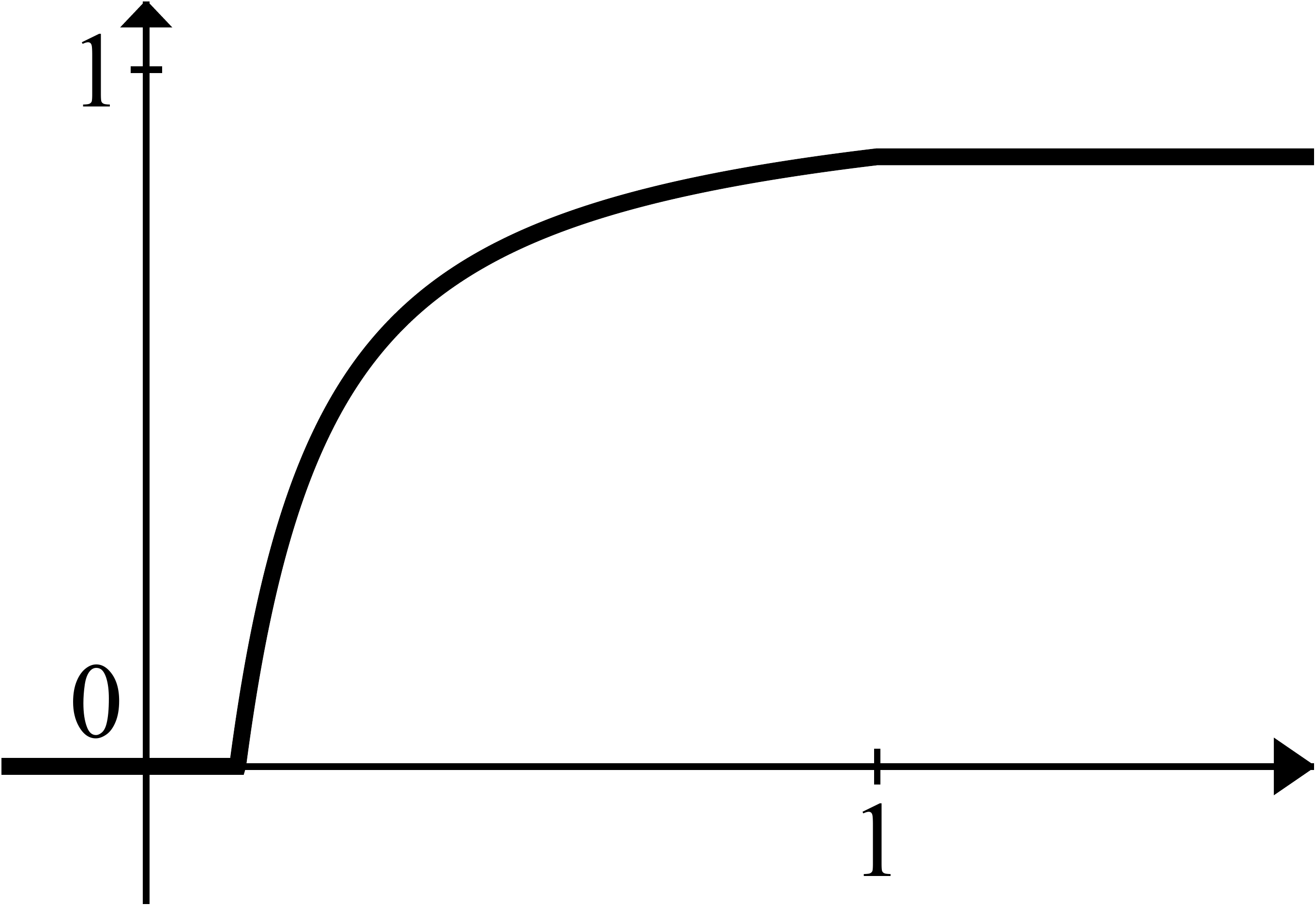}
\end{center}
\caption{The graph of  $\varrho_{8/7}$.\label{fig:rho}}
\end{wrapfigure}
For $\alpha > 1$ define the continuous function $\varrho_{\alpha} \colon \mathbb{R} \to \mathbb{R}$ by
\begin{align*}
\varrho_{\alpha}(t) \coloneqq
\begin{cases}
0 & \text{if} \; t \leq 1 - 1/\alpha\\
1 - (\alpha - 1)/(\alpha t) & \text{if} \; 1- 1/\alpha < t < 1,\\
1/\alpha & \text{if} \; t \geq 1.
\end{cases}\qquad \qquad \qquad \qquad \qquad \qquad \qquad \qquad
\end{align*}
Notice, $\varrho_{\alpha}$ is concave on $[1-1/\alpha,\infty)$ and, on the interval $(1-1/\alpha, 1)$, it is strictly increasing, see \Cref{fig:rho}. Also, for $t \leq 1 - 1/\alpha$, we have that $1 - (\alpha - 1)/(\alpha t) \leq 0$.

\begin{theorem}\label{thm:thm2}
Let $X$ be a Sturmian subshift of slope $\theta$, let $\alpha > 1$ be given, set $\delta \coloneqq ( n^{-t} )_{n \in \mathbb{N}}$ and fix $t \in (1 - 1/\alpha, 1)$.
\begin{enumerate}[label=(\arabic*),leftmargin=2.5em]
\item\label{thm:thm2(1)} 
The metric $d_{s, \delta}$ is sequentially $\overline{\varrho_{\alpha}(t)}$-H\"older regular to the metric $d_{\delta}$ if and only if $\theta \in \overline{\Theta}_{\alpha}$.
\item\label{thm:thm2(2)} 
The metric $d_{s, \delta}$ is sequentially $\underline{\varrho_{\alpha}(t)}$-H\"older regular to the metric $d_{\delta}$ if and only if $\theta \in \underline{\Theta}_{\alpha}$.
\end{enumerate}
Hence, $d_{s, \delta}$ is sequentially $\varrho_{\alpha}(t)$-H\"older regular to $d_{\delta}$ if and only if $\theta \in \Theta_{\alpha}$.
\end{theorem}

\begin{theorem}\label{thm:thm2.5}
Let $X$ be a Sturmian subshift of slope $\theta$, let $\alpha > 1$ be given, set $\delta \coloneqq ( n^{-t} )_{n \in \mathbb{N}}$.
\begin{enumerate}[label=(\arabic*),leftmargin=2.5em]
\item\label{thm:thm2(2)(c)} If $t = 1$, then we have the following.
\begin{enumerate}[label=(\alph*)]
\item\label{thm:thm2(2)(c)2} If $d_{s, \delta}$ is sequentially $\overline{\varrho_{\alpha}(t)}$-H\"older regular to $d_{\delta}$, then $\theta \in \overline{\Theta}_{\alpha}$.
\item\label{thm:thm2(2)(c)3} If $\theta \in \overline{\Theta}_{\alpha}$, then $d_{s, \delta}$ is critically sequentially $\overline{\varrho_{\alpha}(t)}$-H\"older regular to $d_{\delta}$.
\item\label{thm:thm2(2)(c)4} If $\theta \in \underline{\Theta}_{\alpha}$, then $d_{s, \delta}$ is critically sequentially $\underline{\varrho_{\alpha}(t)}$-H\"older regular to $d_{\delta}$.
\end{enumerate}
\item\label{thm:thm2(2)(a)}
If $t > 1$, then we have the following.
\begin{enumerate}[label=(\alph*)]
\item\label{thm:thm2(2)(a)3} If $\theta \in \overline{\Theta}_{\alpha}$, then $d_{s, \delta}$ is sequentially $\overline{\varrho_{\alpha}(t)}$-H\"older regular to $d_{\delta}$.
\item\label{thm:thm2(2)(a)4} If $\theta \in \underline{\Theta}_{\alpha}$, then $d_{s, \delta}$ is sequentially $\underline{\varrho_{\alpha}(t)}$-H\"older regular to $d_{\delta}$.
\end{enumerate}
\item\label{thm:thm2(2)(b)}
\begin{enumerate}[label=(\alph*)]
\item\label{thm:thm2(2)(b)2} If $t \in (1, \alpha/(\alpha - 1))$ and if $d_{s, \delta}$ is sequentially $\overline{\varrho_{\alpha}(t)}$-H\"older regular to $d_{\delta}$, then $\theta \in \overline{\Theta}_{\alpha}$.
\item\label{thm:thm2(2)(b)3} If $t \geq \alpha/(\alpha - 1)$, then $d_{s, \delta}$ is $\varrho_{\alpha}(t)$-H\"older continuous with respect to $d_{\delta}$.
\end{enumerate}
\end{enumerate}
\end{theorem}

\begin{corollary}\label{cor:cor-main}
Let $X$ be a Sturmian subshift of slope $\theta$, let $\alpha > 1$ be given, set $\delta \coloneqq ( n^{-t} )_{n \in \mathbb{N}}$ and fix $t > 1 - 1/\alpha$. If $\theta \in \Theta_{\alpha}$, then the metric $d_{s, \delta}$ is critically sequentially $\varrho_{\alpha}(t)$-H\"older regular to $d_{\delta}$.
\end{corollary}

We conjecture that \Cref{thm:thm2,thm:thm2.5} hold true for $\delta_{n} = \ell(n) n^{-t}$, where $\ell$ is a slowly varying function. See \cite{BGT:1987} for further details on slowly varying functions.

\begin{theorem}\label{thm:thm1}
For  $\alpha > 1$  and $\theta\in [0,1]$ irrational,  the following are equivalent.
\begin{enumerate}[label=(\arabic*),leftmargin=2.5em]
\item\label{thm:thm1(1)} The Sturmian subshift of slope $\theta$ is $\alpha$-repetitive.
\item\label{thm:thm1(2)} The Sturmian subshift  of slope $\theta$ is $\alpha$-repulsive.
\item\label{thm:thm1(5)} The Sturmian subshift  of slope $\theta$ is $\alpha$-finite.
\item\label{thm:thm1(3)} The slope $\theta$ is well-approximable of $\alpha$-type.
\end{enumerate}
\end{theorem}

\begin{remark}\label{rmk:alpha=1}
An analogue of \Cref{thm:thm2,thm:thm1} also holds for the case that $\alpha = 1$, see \cite{KS:2012} and \cite{D:2000} respectively.  In this case, the sequential H\"older regularity is replaced by Lipschitz equivalence and the following conditions are required on the sequence $\delta = (\delta_{n})_{n \in \mathbb{N}}$. The sequence $\delta = (\delta_{n})_{n \in \mathbb{N}}$ is a strictly decreasing null-sequence, and there exist constant $\overline{c}, \underline{c}$, such that $\underline{c} \delta_{n} \leq \delta_{2n}$ and $\delta_{n m} \leq \overline{c} \delta_{n} \delta_{m}$, for all $n, m \in \mathbb{N}$.  Indeed, in \cite{KS:2012} the typical choice for such a sequence is suggested to be $\delta_{n} = \ln(n) n^{-t}$, where $t > 0$, and investigations of spectral metrics when $\delta_{n} = n^{-1}$ are carried out in \cite{JP:2015}.  Thus our choice of the sequence $\delta_{n}$ is a natural choice.  Further, it has been shown in \cite{KS:2012} that if the sequence $\delta_{n}$ is exponentially decreasing, then the metric $d_{s, \delta}$ is Lipschitz equivalent to $d_{\delta}$, independent of the Sturmian subshift.  The latter part of \Cref{thm:thm2.5}~\ref{thm:thm2(2)(b)}\ref{thm:thm2(2)(b)3} gives the counterpart condition to conclude H\"older continuity, independent of the Sturmian subshift.
\end{remark}

\begin{remark}\label{rmk:strengthening}
\Cref{prop:curious,prop:curious-2} give a clear indication that the sequential H\"older regularity in \Cref{thm:thm2,thm:thm2.5} cannot be strengthen to H\"older equivalence.
\end{remark}

A very natural question is if the three combinatorical properties ($\alpha$-repetitive, $\alpha$-repulsive and $\alpha$-finite) are equivalent outside of the setting of Sturmian sequences.  This was one of the main motivations of the sequel to this article \cite{DKMSS:2017}, where in Theorem 3.1 it is shown that in general $\alpha$-repulsive and $\alpha$-finite are equivalent.  Moreover, examples are given which demonstrate that the notions of $\alpha$-repetitive and $\alpha$-repulsive are in fact different.

\begin{theorem}\label{thm:thm3}
For $\alpha > 1$ we have that $\textup{dim}_{\mathcal{H}}(\Theta_{\alpha}) = \textup{dim}_{\mathcal{H}}(\underline{\Theta}_{\alpha}) = 2/(\alpha+1)$ and $\Lambda( \overline{\Theta}_{\alpha} ) = 1$.  (Here, $\textup{dim}_{\mathcal{H}}$ denotes the Hausdorff dimension and $\Lambda$ denotes the one-dimensional Lebesgue measure.)
\end{theorem}

To obtain that $\textup{dim}_{\mathcal{H}}(\Theta_{\alpha}) = 2/(\alpha+1)$,  we show that $\Theta_{\alpha}$ is contained in a countable union of Jarn\'{\i}k sets  each with the same Hausdorff dimension, namely $2/(\alpha+1)$.  We also show that  the exact  Jarn\'{\i}k set  $\textup{Exact}(\alpha+1)$   is contained in $\Theta_{\alpha}$.  Jarn\'{\i}k sets and exact Jarn\'{\i}k sets are defined directly below.   From these observations and the results of \cite{B:2003,B:2008}, one may conclude that $\textup{dim}_{\mathcal{H}}(\Theta_{\alpha}) = 2/(\alpha+1)$.

\begin{definition}\label{defn:exact_Jarnik}
Given a strictly positive monotonically decreasing function $\psi \colon \mathbb{R} \to \mathbb{R}$, the set
\begin{align*}
\mathcal{J}_{\psi} \coloneqq \left\{ x   \in [0, 1] \colon \left\lvert x - \frac{p }{q } \right\rvert \leq \psi(q) \; \text{ for infinitely many} \; p,q\in\mathbb{N} \right\}
\end{align*}
is called the \emph{$\psi$-Jarn\'{\i}k set}.  When $\psi(y) = c  y^{-\beta}$, where $\beta > 2$ and $c > 0$, we denote the set $\mathcal{J}_{\psi}$ by $\mathcal{J}_{\beta}^{c}$ and define
\begin{align*}
\textup{Exact}(\beta) \coloneqq \mathcal{J}_{\beta}^{1} \setminus \bigcup_{n \geq 2, \, n \in \mathbb{N}} \mathcal{J}_{\beta}^{n/(n+1)}
\end{align*}
to be the set of real numbers that are approximable to rational numbers $p/q$ to order $q^{\beta}$ but no better.
\end{definition}

\begin{theorem}[\cite{B:2012,BBDV:2009,B:2003,B:2008}]\label{rm.:Jarnik_dimension}
For $\beta > 2$ and $c > 0$, we have that $\textup{dim}_{\mathcal{H}}(\mathcal{J}_{\beta}^{c}) = \textup{dim}_{\mathcal{H}}(\textup{Exact}(\beta)) = 2/\beta$.
\end{theorem}

Notice, by \Cref{prop:theta-1-theta} and \Cref{rmk:eta}, it is sufficient to prove the above results (\Cref{thm:thm2,thm:thm2.5,thm:thm1,thm:thm3}) for $\theta \in [0, 1/2]$, and so, from here on, we assume that $\theta = [0; a_{1}+1, a_{2}, \dots ]  \in [0, 1/2]$ with $a_{n} \in \mathbb{N}$ and $n \in \mathbb{N}$.

\section{Preliminaries}

\subsection{Aperiodic order}\label{sec:Prelim_AO}

In the following, let $\tau$ and $\rho$  denote the semigroup homomorphisms on $\{ 0, 1\}^{*}$ determined by
$\tau(0) \coloneqq (0)$,
$\tau(1) \coloneqq (1,0)$,
$\rho(0) \coloneqq (0, 1)$,
and
$\rho(1) \coloneqq (1)$.
For $\theta = [0; a_{1} + 1, a_{2}, \dots] \in [0, 1/2]$ irrational, we set $\mathcal{R}_{0} = \mathcal{R}_{0}(\theta) \coloneqq (0)$, $\mathcal{L}_{0} = \mathcal{L}_{0}(\theta) \coloneqq (1)$ and, for $k \in \mathbb{N}$,
\begin{align*}
\mathcal{R}_{k} = \mathcal{R}_{k}(\theta) \coloneqq \tau^{a_{1}} \rho^{a_{2}} \tau^{a_{3}} \rho^{a_{4}} \cdots \tau^{a_{2k - 1}} \rho^{a_{2k}}(0)
\quad \text{and} \quad
\mathcal{L}_{k} = \mathcal{L}_{k}(\theta) \coloneqq \tau^{a_{1}} \rho^{a_{2}} \tau^{a_{3}} \rho^{a_{4}} \cdots \tau^{a_{2k - 1}} \rho^{a_{2k}}(1).
\end{align*}

\begin{theorem}[\cite{AR:1991}]\label{thm:AR-1991}
Let $X$ denote a Sturmian subshift of slope $\theta$.  Let $x$, $y$ denote the unique infinite words  with $x\lvert_{\lvert \mathcal{R}_{k} \rvert} = \mathcal{R}_{k}$ and $y\lvert_{\lvert \mathcal{L}_{k} \rvert} = \mathcal{L}_{k}$, for all $k \in \mathbb{N}$.   The words $x$ and $y$ belong to $X$, and hence, by the minimality of a Sturmian subshift, $X = \Omega(x) = \Omega(y)$.
\end{theorem}

\begin{proposition}\label{prop:blackboard}
Let $\theta = [0; a_{1} + 1, a_{2}, \dots] \in [0, 1/2]$ be an irrational number. For all $k \in \mathbb{N}$ we have
\begin{align*}
\begin{aligned}
\lvert R_{k} \rvert &= q_{2k},\\[0.25em]
\mathcal{R}_{k} &= \mathcal{R}_{k-1} \underbrace{\mathcal{L}_{k} \dots \mathcal{L}_{k}}_{a_{2k}},
\end{aligned}
\qquad \qquad
\begin{aligned} 
\lvert \mathcal{L}_{k} \rvert &= q_{2k - 1},\\[0.25em]
\mathcal{L}_{k} &= \mathcal{L}_{k-1} \underbrace{\mathcal{R}_{k-1} \dots \mathcal{R}_{k-1}}_{a_{2k-1}}.
\end{aligned}
\end{align*}
\end{proposition}

\begin{proof}
An inductive argument together with the definitions of $\mathcal{R}_{k}$ and $\mathcal{L}_{k}$ yields the required result.
\end{proof}

\begin{corollary}\label{for:rightspecial}
Let $\theta = [0; a_{1} + 1, a_{2}, \dots ] \in [0, 1/2]$ be irrational,  $k \in \mathbb{N}_{0}$,  $n \in \{ 0, 1, \dots, a_{2(k+1)} - 1 \}$ and $m \in \{ 0, 1, \dots, a_{2(k+1)-1} - 1 \}$.  The words $\displaystyle \mathcal{R}_{k} \underbrace{\mathcal{L}_{k+1} \dots \mathcal{L}_{k+1}}_{n}$ and $\displaystyle \mathcal{L}_{k + 1} \underbrace{\mathcal{R}_{k} \dots \mathcal{R}_{k}}_{m}$ are right special.
\end{corollary}

\begin{proof}
An application of \Cref{prop:blackboard} and \Cref{thm:AR-1991} yields that $\mathcal{R}_{k}$ and $\mathcal{L}_{k+1}$ are right special words for all $k \in \mathbb{N}_{0}$.  For all $n \in \{ 1, \dots, a_{2(k+1)} - 1 \}$ and $m \in \{ 1, \dots, a_{2(k+1)-1}-1 \}$, we observe
\begin{align*}
\sigma^{\lvert \mathcal{L}_{k} \rvert + a_{2(k+1)-1} \lvert \mathcal{R}_{k} \rvert + (a_{2(k+1)} - (n + 1))\lvert \mathcal{L}_{k+1}\rvert}(\mathcal{R}_{k+1})
&= \mathcal{R}_{k} \underbrace{\mathcal{L}_{k+1} \dots \mathcal{L}_{k+1}}_{n},\\[1em]
\sigma^{\lvert \mathcal{R}_{k-1} \rvert + a_{2k} \lvert \mathcal{L}_{k} \rvert + (a_{2(k+1)-1} - (m + 1)) \lvert \mathcal{R}_{k} \rvert}(\mathcal{L}_{k+1})
&= \mathcal{L}_{k} \underbrace{\mathcal{R}_{k} \dots \mathcal{R}_{k}}_{m}.
\end{align*}
The above in tandem with \Cref{rmk:BFMS-2002} and \Cref{prop:blackboard} yields the result.
\end{proof}

\begin{corollary}\label{cor:branching_points}
Let $X$ be a Sturmian subshift of slope $\theta = [0; a_{1} + 1, a_{2}, \dots ] \in [0, 1/2]$.  If $x, y \in X$ are the unique infinite words such that $x\lvert_{\lvert \mathcal{R}_{m} \rvert} = \mathcal{R}_{m}$ and $y\lvert_{\lvert \mathcal{L}_{m} \rvert} = \mathcal{L}_{m}$, for all $m \in \mathbb{N}$, then
\begin{enumerate}[label=(\arabic*),leftmargin=2.5em]
\item\label{enumerate:codition_1_structure} $\overline{b}_{n}(x) = 1$ if and only if $n = j q_{2 k - 1} + q_{2 k - 2}$ for some $k \in \mathbb{N}$ and $j \in \{ 0, 1, \dots, a_{2k} - 1 \}$, and 
\item\label{enumerate:codition_2_structure} $\overline{b}_{m}(y) = 1$ if and only if $m = i q_{2 l} + q_{2 l - 1}$ for some $l \in \mathbb{N}$ and $i \in \{ 0, 1, \dots, a_{2k + 1}-1 \}$.
\end{enumerate}
\end{corollary}

\begin{proof}
\Cref{for:rightspecial} gives the reverse implication: if $n = j q_{2 k - 1} + q_{2 k - 2}$, for some $k \in \mathbb{N}$ and some $j \in \{ 0, 1, \dots, a_{2k} - 1 \}$, then $\overline{b}_{n}(x) = 1$, and if $m = i q_{2 l} + q_{2 l - 1}$, for some $l \in \mathbb{N}$ and $i \in \{ 0, 1, \dots, a_{2l + 1}-1 \}$, then $\overline{b}_{m}(y) = 1$.

For the forward implication, we show the result for $b_{n}(x)$ and $b_{m}(y)$ where $n \leq \lvert\mathcal{R}_{1} \rvert = q_{2}$ and where $m \leq \lvert L_{2} \rvert = q_{3}$ after which we proceed by induction to obtain the general result.

By \Cref{rmk:BFMS-2002} and \Cref{for:rightspecial} it follows that $b_{1}(x) = 1$ and, for $m \in \{ 1, 2, \dots, q_{1} - 1 \}$, that $b_{m}(y) = 0$.  Consider the word $\mathcal{R}_{1} = x\lvert_{\lvert \mathcal{R}_{1} \rvert} = x\lvert_{ q_{2}}$.  Let $n = k q_{1} + (j + 1)q_{0}$ for some $k \in \{ 0, 1, \dots, a_{2} - 1 \}$ and some $j \in \{ 1, 2, \dots, a_{1} \}$.  For $k = 0$,
\begin{align*}
x\lvert_{n} = \mathcal{R}_{1}\lvert_{n} = (0, 1, \underbrace{0, 0, \dots, 0}_{j - 1} ).
\end{align*}
By \Cref{prop:blackboard} and \Cref{for:rightspecial},
\begin{align*}
\mathcal{L}_{1} = (1, \underbrace{0, 0, \dots, 0}_{a_{1}} )
\end{align*}
is a right special word and thus, by \Cref{rmk:BFMS-2002}, the set of all right special words of length at most $\lvert L_{1} \rvert = a_{1} + 1$ is
\begin{align*}
\{ (1, \underbrace{0, 0, \dots, 0}_{a_{1}} ), (\underbrace{0, 0, \dots, 0}_{a_{1}} ), (\underbrace{0, 0, \dots, 0}_{a_{1}-1} ), \dots, (0, 0), (0) \}.
\end{align*}
Since there exists a unique right special word per length, it follows that $b_{n}(x) = 0$.  In the case that $k \in \{ 1, \dots, a_{2} - 1 \}$,
\begin{align*}
\sigma^{n - \lvert L_{1} \rvert}(x\lvert_{n}) = ( \underbrace{0, 0, \dots, 0}_{a_{1} - (j - 1)}, 1, \underbrace{0, 0, \dots, 0}_{j - 1} ),
\end{align*}
where we recall that $a_{1} - (j - 1) \geq 1$.  Since there exists a unique right special word per length and since
\begin{align*}
\lvert\sigma^{n - \lvert L_{1} \rvert}(x\lvert_{n}) \rvert = \lvert \sigma^{n - \lvert L_{1} \rvert}( \underbrace{0, 0, \dots, 0}_{a_{1} - (j - 1)}, 1, \underbrace{0, 0, \dots, 0}_{j - 1} ) \rvert = \lvert L_{1} \rvert,
\end{align*}
it follows that $b_{n}(x) = 0$.  An application of \Cref{for:rightspecial} completes the proof for $n \leq \lvert\mathcal{R}_{1} \rvert = q_{2}$.

Consider the word $\mathcal{L}_{2} = y\lvert_{\lvert \mathcal{L}_{2} \rvert} = y\lvert_{ q_{3}}$.  Let $m = l q_{2}  + 1 + (i + 1) q_{1} = l \lvert \mathcal{R}_{1} \rvert  + 1 + (i + 1) \lvert \mathcal{L}_{1} \rvert $  for some $l \in \{ 0, 1, \dots, a_{3} - 1 \}$ and $i \in \{ 0, 1, \dots, a_{2}-1 \}$.  By \Cref{prop:blackboard} we have that
\begin{align*}
\sigma^{l \lvert \mathcal{R}_{1} \rvert +1}(y\lvert_{m}) 
= \sigma^{l \lvert \mathcal{R}_{1} \rvert +1}(\mathcal{L}_{2}\lvert_{m})
= \sigma( \mathcal{L}_{1} \mathcal{R}_{0} \underbrace{\mathcal{L}_{1}\mathcal{L}_{1} \dots \mathcal{L}_{1}}_{i} )
=  \underbrace{\mathcal{R}_{0} \mathcal{R}_{0} \dots \mathcal{R}_{0}}_{a_{1} + 1 = q_{1} = \lvert \mathcal{L}_{1}\rvert} \underbrace{\mathcal{L}_{1} \mathcal{L}_{1} \dots \mathcal{L}_{1}}_{i}
\end{align*}
and hence $\lvert \sigma^{l \lvert \mathcal{R}_{1} \rvert +1}(y\lvert_{m}) \rvert = (i + 1) \lvert \mathcal{L}_{1} \rvert = (i + 1) q_{1}$.  By \Cref{rmk:BFMS-2002} and \Cref{for:rightspecial},
\begin{align*}
\sigma^{1 + (a_{2} - (i + 1))\lvert \mathcal{L}_{1} \rvert}(x\lvert_{q_{2}})
= \sigma^{1 + (a_{2} - (i + 1))q_{1}}(x\lvert_{q_{2}})
= \sigma^{1 + (a_{2} - (i + 1))q_{1}}(\mathcal{R}_{1})
= \underbrace{\mathcal{L}_{1} \mathcal{L}_{1} \dots \mathcal{L}_{1}}_{i + 1}
\end{align*}
is a right special word of length $(i + 1)\lvert \mathcal{L}_{1} \rvert = (i + 1) q_{1}$.  Since there is a unique right special word per length and since 
\begin{align*}
\underbrace{\mathcal{R}_{0} \mathcal{R}_{0} \dots \mathcal{R}_{0}}_{a_{1} + 1 = q_{1} = \lvert \mathcal{L}_{1}\rvert} \underbrace{\mathcal{L}_{1} \mathcal{L}_{1} \dots \mathcal{L}_{1}}_{i}
\neq
\underbrace{\mathcal{L}_{1} \mathcal{L}_{1} \dots \mathcal{L}_{1}}_{i + 1}
\end{align*}
it follows that $b_{m}(y) = 0$.  An application of \Cref{for:rightspecial} yields the result for $m \leq \lvert \mathcal{L}_{2} \rvert = q_{3}$.

Assume there is $r \in \mathbb{N}$ so that the result holds for all natural numbers $n < q_{2r}$ and $m < q_{2r+1}$, namely,
\begin{enumerate}[label=(\roman*),leftmargin=2.5em]
\item\label{induction(a)} $\overline{b}_{n}(x) = 1$ if and only if $n = j q_{2 k - 1} + q_{2 k - 2}$ for $k \in \{ 1, 2, \dots, r \}$ and $j \in \{ 0, 1, \dots, a_{2k} - 1 \}$, and 
\item\label{induction(b)} $\overline{b}_{m}(y) = 1$ if and only if $m = i q_{2 l} + q_{2 l - 1}$ for $l \in \{ 1, 2, \dots, r \}$ and $i \in \{ 0, 1, \dots, a_{2l + 1}-1 \}$.
\end{enumerate}
The proof of \ref{induction(a)} and \ref{induction(b)} for $r+1$ follow in the same manner; thus below we provide the proof of \ref{induction(a)} for $r+1$ and leave the proof of \ref{induction(b)} to the reader.  To this end consider the word
\begin{align*}
x\lvert_{\lvert \mathcal{R}_{r+1}\rvert} = \mathcal{R}_{r+1} = \mathcal{R}_{r} \underbrace{\mathcal{L}_{r+1} \mathcal{L}_{r+1} \dots \mathcal{L}_{r+1}}_{a_{2(r+1)}}.
\end{align*}
By way of contradiction, suppose there exists an integer $n$ such that $\lvert \mathcal{R}_{r} \rvert < n \leq \lvert \mathcal{R}_{r+1} \rvert$, $n$ is not of the form stated in Part~\ref{enumerate:codition_1_structure} and $b_{n}(x) = 1$.  For if not, the result is a consequence of \Cref{for:rightspecial}.  By our hypothesis, we have that, $n = \lvert \mathcal{R}_{r} \rvert + (a_{2(r+1)} - 1 - b) \lvert \mathcal{L}_{r+1} \rvert + \lvert \mathcal{L}_{r} \rvert + (a_{2(r+1)-1} - a) \lvert \mathcal{R}_{r} \rvert$, where $a \in \{ 1, 2, \dots a_{2(r+1)-1}  \}$ and $b \in \{ 0, 1, \dots , a_{2(r+1)} -1 \}$.  Set
\begin{align*}
v = \mathcal{R}_{r} \underbrace{\mathcal{L}_{r+1} \mathcal{L}_{r+1} \dots \mathcal{L}_{r+1}}_{a_{2(r+1)}-1-b} \mathcal{L}_{r} \underbrace{\mathcal{R}_{r} \mathcal{R}_{r} \dots \mathcal{R}_{r}}_{a_{2(r+1)-1}-a}
\quad \text{and} \quad
w = \underbrace{\mathcal{R}_{r} \mathcal{R}_{r} \dots \mathcal{R}_{r}}_{a} \underbrace{\mathcal{L}_{r+1} \mathcal{L}_{r+1} \dots \mathcal{L}_{r+1}}_{b},
\end{align*}
so that $\lvert v \rvert = n$, $\lvert w \rvert = \mathcal{R}_{r+1} - n$, $x\lvert_{\lvert \mathcal{R}_{r+1}} = \mathcal{R}_{r+1} = vw$ and $\lvert \sigma^{\lvert w \rvert}(\mathcal{R}_{r+1}) \rvert = \lvert v \rvert$.  \Cref{for:rightspecial} implies $\sigma^{\lvert w \rvert}(x\lvert_{\lvert \mathcal{R}_{r+1}}) = \sigma^{\lvert w \rvert}(\mathcal{R}_{r+1})$ is a right special word.  Since we have assumed that $b_{n}(x) = 1$ and since there exists a unique right special word per length (\Cref{rmk:BFMS-2002}) it follows that $\sigma^{\lvert w \rvert}(\mathcal{R}_{r+1}) = v$.  If $a = 1$, then 
\begin{align*}
\sigma^{\lvert w \rvert}(\mathcal{R}_{r+1})
= \sigma^{\lvert \mathcal{R}_{r} \rvert + b \lvert \mathcal{L}_{r+1} \rvert}(\mathcal{R}_{r} \underbrace{\mathcal{L}_{r+1} \mathcal{L}_{r+1} \dots \mathcal{L}_{r+1}}_{a_{2(r+1)}})
= \underbrace{\mathcal{L}_{r+1} \mathcal{L}_{r+1} \dots \mathcal{L}_{r+1}}_{a_{2(r+1)} - b}.
\end{align*}
This is a contradiction to the assumption $b_{n}(x) = 1$; since if this were the case we would have that $\sigma^{\lvert w \rvert}(\mathcal{R}_{r+1}) = v$, but the first letter of $v$ is $0$ and the first letter of $\sigma^{\lvert w \rvert}(\mathcal{R}_{r+1})$ is $1$.  Hence, $a \geq 2$, and so
\begin{align*}
\sigma^{\lvert w \rvert}(\mathcal{R}_{r+1})
&= \sigma^{a \lvert \mathcal{R}_{r} \rvert + b \lvert \mathcal{L}_{r+1} \rvert}(\mathcal{R}_{r} \underbrace{\mathcal{L}_{r+1} \mathcal{L}_{r+1} \dots \mathcal{L}_{r+1}}_{a_{2(r+1)}})\\
&= \sigma^{(a - 1) \lvert \mathcal{R}_{r} \rvert}(\mathcal{L}_{r} \underbrace{\mathcal{R}_{r} \mathcal{R}_{r} \dots \mathcal{R}_{r}}_{a_{2(r+1) - 1}}\underbrace{\mathcal{L}_{r+1} \mathcal{L}_{r+1} \dots \mathcal{L}_{r+1}}_{a_{2(r+1)}-b - 1})\\
&= \sigma^{\lvert \mathcal{R}_{r} \rvert - \lvert \mathcal{L}_{r} \rvert}(\underbrace{\mathcal{R}_{r} \mathcal{R}_{r} \dots \mathcal{R}_{r}}_{a_{2(r+1) - 1} - (a - 2)}\underbrace{\mathcal{L}_{r+1} \mathcal{L}_{r+1} \dots \mathcal{L}_{r+1}}_{a_{2(r+1)}-b - 1})\\
&= \sigma^{(a_{2r}-1)\lvert \mathcal{L}_{r} \rvert + \lvert \mathcal{R}_{r-1} \rvert}(\mathcal{R}_{r-1} \underbrace{\mathcal{L}_{r} \mathcal{L}_{r} \dots \mathcal{L}_{r}}_{a_{2r}} \hspace{-0.5em}\underbrace{\mathcal{R}_{r} \mathcal{R}_{r} \dots \mathcal{R}_{r}}_{a_{2(r+1) - 1} - (a - 2)-1}\hspace{-0.5em}\underbrace{\mathcal{L}_{r+1} \mathcal{L}_{r+1} \dots \mathcal{L}_{r+1}}_{a_{2(r+1)}-b - 1})\\
&= \mathcal{L}_{r} \hspace{-0.7em}\underbrace{\mathcal{R}_{r} \mathcal{R}_{r} \dots \mathcal{R}_{r}}_{a_{2(r+1) - 1} - (a - 2)-1}\hspace{-0.7em}\underbrace{\mathcal{L}_{r+1} \mathcal{L}_{r+1} \dots \mathcal{L}_{r+1}}_{a_{2(r+1)}-b - 1}),
\end{align*}
where we observe $a_{2(r+1)} - (a-2)-1 \geq 1$ and $a_{2(r+1)} - b - 1 \geq 0$. This contradicts the assumption $b_{n}(x) = 1$; since if this were the case we would have $\sigma^{\lvert w \rvert}(\mathcal{R}_{r+1}) = v$, but the first letter of $v$ is $0$ and the first letter of $\sigma^{\lvert w \rvert}(\mathcal{R}_{r+1})$ is $1$.
\end{proof}

\subsection{Spectral metrics}\label{sec:Pre_lim_Spectral_Metric}

Let $X$ denote a Sturmian subshift of slope $\theta = [0; a_{1} + 1, a_{2}, \dots ] \in \Theta_{\alpha} \cap [0, 1/2]$ and let $x, y \in X$ denote the unique infinite words such that $x\lvert_{\lvert \mathcal{R}_{n} \rvert} = \mathcal{R}_{n}$ and $y\lvert_{\lvert \mathcal{L}_{n} \rvert} = \mathcal{L}_{n}$, for all $n \in \mathbb{N}$.  By \Cref{prop:blackboard}, we have, for all $n \in \mathbb{N}$, that
\begin{align}\label{eq:unltrametric_shift_words}
\begin{aligned}
\sigma^{\lvert \mathcal{L}_{n} \rvert}(y)\lvert_{\lvert \mathcal{R}_{n} \rvert + 1} &= \mathcal{R}_{n} (0),\\[0.25em]
d_{\delta}(x, \sigma^{\lvert \mathcal{L}_{n} \rvert}(y)) &= \delta_{q_{2n}},
\end{aligned}
\qquad \qquad
\begin{aligned}
\sigma^{\lvert \mathcal{R}_{n} \rvert}(x)\lvert_{\lvert \mathcal{L}_{n+1} \rvert + 1}  &= \mathcal{L}_{n+1}(1),\\[0.25em]
d_{\delta}(\sigma^{\lvert \mathcal{R}_{n} \rvert}(x), y) &= \delta_{q_{2(n+1)-1}}.
\end{aligned}
\end{align}
Combining \Cref{cor:branching_points} and \eqref{eq:spectral_metric_bound}, we obtain that
\begin{align}\label{eq:xyL}
\begin{aligned}
d_{s, \delta}(x, \sigma^{\lvert \mathcal{L}_{n} \rvert}(y))
&= \sum_{k = 2n}^{\infty} \sum_{j = 1}^{a_{k+1}} \delta_{j q_{k} + q_{k - 1} - \mathds{1}_{2\mathbb{Z}}(k-2n) q_{2n-1}},\\
d_{s, \delta}(\sigma^{\lvert \mathcal{R}_{n} \rvert}(x), y)
&= \sum_{k = 2n+1}^{\infty} \sum_{j = 1}^{a_{k+1}} \delta_{j q_{k} + q_{k - 1} - \mathds{1}_{2\mathbb{Z}}(k - (2n+1)) q_{2n}},
\end{aligned}
\end{align}
where $\mathds{1}_{2 \mathbb{Z}}$ denotes the characteristic function on the group $2 \mathbb{Z}$ of even integers.  For $r > 0$, we set
\begin{align*}
\psi_{x, n}(r) \coloneqq \frac{d_{s, \delta}(x, \sigma^{\lvert \mathcal{L}_{n} \rvert}(y))}{d_{\delta}(x, \sigma^{\lvert \mathcal{L}_{n} \rvert}(y))^{r}}
\quad \text{and} \quad
\psi_{y, n}(r) \coloneqq \frac{d_{s, \delta}(\sigma^{\lvert \mathcal{R}_{n} \rvert}(x), y)}{d_{\delta}(\sigma^{\lvert \mathcal{R}_{n} \rvert}(x), y)^{r}}.
\end{align*}
Notice that $\displaystyle{\limsup_{n \to \infty} \psi_{z, n}(r) \leq \psi_{z}(r)} \leq \psi(r)$ for $z\in\{x,y\}$. 

\begin{proposition}\label{prop:curious}
Let $\alpha > 1$ and let $X$ denote a Sturmian subshift of slope $\theta \in [0, 1/2]$.  Let $t > 1 - 1/\alpha$ and set $\delta \coloneqq ( n^{-t} )_{n \in \mathbb{N}}$.
\begin{enumerate}[label=(\arabic*),leftmargin=2.5em]
\item\label{prop:curious_1} 
\begin{enumerate}[label=(\alph*)]
\item\label{prop:curious_1a} If $t \in (1-1/\alpha,1)$ and $A_{\alpha}(\theta) < \infty$, then
\begin{align*}
\sup_{z \in \{ x, y \}} \limsup_{n \to \infty} \psi_{z, n}(r) 
\begin{cases}
= 0 & \text{if} \; 0 < r < \alpha - (\alpha - 1)/t),\\
< \infty & \displaystyle \text{if} \; r = \alpha - (\alpha - 1)/t.
\end{cases}
\end{align*}
\item\label{prop:curious_1b} If $t \in (1-1/\alpha,1)$ and $A_{\alpha}(\theta) > 0$, then
\begin{align*}
\sup_{z \in \{ x, y \}} \limsup_{n \to \infty} \psi_{z, n}(r)
\begin{cases}
= \infty & \displaystyle \text{if} \; r > \alpha - (\alpha - 1)/t,\\
> 0 & \text{if} \; r = \alpha - (\alpha - 1)/t).
\end{cases}
\end{align*}
\end{enumerate}
\item\label{prop:curious_2} 
\begin{enumerate}[label=(\alph*)]
\item\label{prop:curious_2a} If $t = 1$, $A_{\alpha}(\theta) < \infty$ and $r \in (0, 1)$, then
\begin{align*}
\sup_{z \in \{ x, y \}} \limsup_{n \to \infty} \psi_{z, n}(r) = 0,
\end{align*}
\item\label{prop:curious_2b} If $t = 1$ and if $r \geq 1$, then
\begin{align*}
\sup_{z \in \{ x, y \}} \limsup_{n \to \infty} \psi_{z, n}(r) = \infty.
\end{align*}
\end{enumerate}
\item\label{prop:curious_3} 
\begin{enumerate}[label=(\alph*)]
\item\label{prop:curious_3a} If $t > 1$, then
\begin{align*}
\sup_{z \in \{ x, y \}} \limsup_{n \to \infty} \psi_{z, n}(r) 
\begin{cases}
= 0 & \text{if} \; 0 < r < 1,\\
< \infty & \displaystyle \text{if} \; r =1.
\end{cases}
\end{align*}
\item\label{prop:curious_3b} If $t \in (1-1/\alpha,1)$, then
\begin{align*}
\sup_{z \in \{ x, y \}} \limsup_{n \to \infty} \psi_{z, n}(r)
\begin{cases}
= \infty & \displaystyle \text{if} \; r > 1,\\
> 0 & \text{if} \; r = 1.
\end{cases}
\end{align*}
\end{enumerate}
\end{enumerate}
\end{proposition}

\begin{remark}\label{rmk:exp_q_k}
In the proof of all three parts of \Cref{prop:curious}, we will use the following observation. From the iterative definition of the sequence $(q_{k})_{k \in \mathbb{N}}$ and using an inductive argument, we have that $q_{k+j} > f_{j+1}  q_{k}$, for all $k \in \mathbb{N}$ and $j \geq 0$.  Here, $f_{k}$ denotes the $k$-th Fibonacci number, that is, $f_{1} = 1$, $f_{2} = 1$ and $f_{k+1} = f_{k} + f_{k-1}$.  Setting $\gamma \coloneqq (1 + \sqrt{5})/2$, it is known that $f_{k} = (\gamma^{k} - (- \gamma)^{-k})/\sqrt{5}$ and so, $f_{k} > \gamma^{k}/(2 \sqrt{5})$.  Thus, we have $q_{k + j} > q_{k}\gamma^{j}/(2 \sqrt{5})$, for $k \in \mathbb{N}$ and $j \geq 0$. 
\end{remark}

\begin{proof}[Proof of \Cref{prop:curious}~\ref{prop:curious_1}\ref{prop:curious_1a}]
Since $A_{\alpha}(\theta) < \infty$, there exists a constant $c > 1$ with $a_{k+1}  q_{k}^{1- \alpha} < c$, for all sufficiently large $k \in \mathbb{N}$.  This with \Cref{rmk:exp_q_k} and the fact $t \in (1-1/\alpha,1)$ yields the following.
\begin{align}\label{eq:copy-paste-argument}
\begin{aligned}
&\limsup_{m \to \infty} q_{m}^{t  r}  \sum_{k = m}^{\infty}  \sum_{j = 1}^{a_{k+1}} \frac{1}{(j  q_{k} + q_{k-1} - \mathds{1}_{2\mathbb{Z}}(k-m)q_{m-1})^{t}}\\
&\leq \limsup_{m \to \infty} q_{m}^{t  r}  \sum_{k = m}^{\infty}  \frac{1}{q_{k}^{t}}  \sum_{j = 1}^{a_{k+1}} \frac{1}{j^{t}}\\
&\leq \limsup_{m \to \infty} \frac{1 + 2^{t}}{1 - t} q_{m}^{t  r}  \sum_{k = m}^{\infty}  \frac{a_{k+1}^{1-t}}{q_{k}^{t}}\\
&= \limsup_{m \to \infty} \frac{1 + 2^{t}}{1 - t} q_{m}^{t  r}  \sum_{k = m}^{\infty}  \left( \frac{a_{k+1}}{q_{k}^{\alpha - 1}}  \right)^{1-t} \frac{1}{q_{k}^{1 - \alpha  (1-t)}}\\
&\leq \limsup_{m \to \infty} \frac{(1 + 2^{t})  c^{1 - t}}{1 - t} q_{m}^{t  r}  \sum_{k = m}^{\infty}  \frac{1}{q_{k}^{1 - \alpha  (1-t)}}\\
&\leq \limsup_{m \to \infty} \frac{(1 + 2^{t})  (2 \sqrt{5})^{1 - \alpha (1-t)}  c^{1 - t}}{1 - t} q_{m}^{t  r}  \sum_{j = 0}^{\infty}  \frac{1}{q_{m}^{1 - \alpha  (1-t)} \gamma^{j (1 - \alpha (1-t))}}\\
&= \limsup_{m \to \infty} \frac{(1 + 2^{t})  (2  \sqrt{5})^{1 - \alpha (1-t)}  c^{1 - t}}{(1 - t)(1- \gamma^{-(1 - \alpha (1-t))})} q_{m}^{t  r - 1 + \alpha  (1-t)}.
\end{aligned}
\end{align}
This latter value is equal to zero if $0 < r < \alpha - (\alpha - 1)/t$, and finite if $r = \alpha - (\alpha - 1)/t$. This together with \eqref{eq:unltrametric_shift_words} and \eqref{eq:xyL} yields that, for $r \in (0, \alpha - (\alpha - 1)/t)$,
\begin{align*}
\sup_{z \in \{ x, y \}} \limsup_{n \to \infty} \psi_{z, n}(\alpha - (\alpha - 1)/t) < \infty
\quad \text{and} \quad
\sup_{z \in \{ x, y \}} \limsup_{n \to \infty} \psi_{z, n}(r) = 0.
\end{align*}
This completes the proof.
\end{proof}

\begin{proof}[Proof of \Cref{prop:curious}~\ref{prop:curious_1}\ref{prop:curious_1b}]
Since $A_{\alpha}(\theta) > 0$ and since $(q_{n})_{n \in \mathbb{N}}$ is an unbounded monotonic sequence, there exists a sequence of natural numbers $(n_{k})_{k \in \mathbb{N}}$, such that
$a_{n_{k}} q_{n_{k}}^{1-\alpha} > A_{\alpha}(\theta)/2$,
for all $k \in \mathbb{N}$.  Hence, we have the following chain of inequalities.
\begin{align*}
\limsup_{m \to \infty} q_{m}^{t  r}  \sum_{j = 1}^{a_{m+1}} \frac{1}{(j  q_{m} )^{t}}
\geq \limsup_{m \to \infty} q_{m}^{t  r - t}  \sum_{j = 1}^{a_{m+1}} \frac{1}{j^{t}}
&\geq \limsup_{m \to \infty} q_{m}^{t  r - t}  \frac{a_{m+1}^{1-t} - 1}{1-t}\\
&\geq \left(\frac{A_{\alpha}(\theta)}{2 (1- t)}\right)^{1-t} \!\! \limsup_{j \to \infty} q_{n_{j}}^{t  r - t + (1-t)(\alpha - 1)}
\end{align*}
This latter term is positive and finite if $r = \alpha - (\alpha - 1)/t$ and is infinite if $r > \alpha - (\alpha - 1)/t$.  Combining this with \eqref{eq:unltrametric_shift_words} and \eqref{eq:xyL} yields the required result.
\end{proof}

\begin{proof}[Proof of \Cref{prop:curious}~\ref{prop:curious_2}\ref{prop:curious_2a}]
Since $A_{\alpha}(\theta) < \infty$, there exists a constant $c > 1$ so that $a_{k+1}  q_{k}^{1- \alpha} < c$, for all $k \in \mathbb{N}$.  We recall that the sequence $(q_{k})_{k \in \mathbb{N}}$ is strictly increasing and notice, for $x > \mathrm{e}^{1}$, that the function $x \mapsto \ln(x)/x$ is strictly decreasing.  Combining these observations with \Cref{rmk:exp_q_k} yields the following chain of inequalities.
\begin{align*}
&\limsup_{m \to \infty}  q_{m}^{r}  \sum_{k = m}^{\infty}  \sum_{j = 1}^{a_{k+1}} \frac{1}{j  q_{k} + q_{k-1} - \mathds{1}_{2\mathbb{Z}}(k-m) q_{m-1}}\\
&\leq \limsup_{m \to \infty}  q_{m}^{r}  \sum_{k = m}^{\infty}  \frac{1}{q_{k}}  \sum_{j = 1}^{a_{k+1}} \frac{1}{j}
\leq \limsup_{m \to \infty}  q_{m}^{r}  \sum_{k = m}^{\infty}  \frac{\ln(a_{k+1}) + 1}{q_{k}}\\
&\leq \limsup_{m \to \infty}  q_{m}^{r-1}  (\ln(c) + (\alpha - 1)\ln(q_{m}) + 1) + q_{m}^{r}  \sum_{k = m+1}^{\infty}  \frac{\ln(c) + (\alpha - 1)\ln(q_{k}) + 1}{q_{k}}\\
&\leq \limsup_{m \to \infty}  q_{m}^{r-1}  (\ln(c) + (\alpha - 1)\ln(q_{m}) + 1) + 2  \sqrt{5}  q_{m}^{r-1}  \sum_{j = 1}^{\infty}  \frac{\ln(c) + (\alpha - 1) j \ln(\gamma) - (\alpha-1) \ln(2 \sqrt{5})+ 1}{\gamma^{j}}
\end{align*}
For $r \in (0, 1)$ this latter value is zero, and thus, by \eqref{eq:unltrametric_shift_words} and \eqref{eq:xyL}, we have 
\!$\displaystyle \sup_{z \in \{ x, y \}} \limsup_{n \to \infty} \psi_{z, n}(r) = 0$.
\end{proof}

\begin{proof}[Proof of \Cref{prop:curious}~\ref{prop:curious_2}\ref{prop:curious_2b}]
If $r \geq 1$, then we have that 
\begin{align*}
\limsup_{m \to \infty} q_{m}^{r}  \sum_{j = 1}^{a_{m+1}} \frac{1}{j  q_{m}}
\geq \limsup_{m \to \infty} q_{m}^{r-1} \sum_{j = 1}^{a_{m+1}} \frac{1}{j} 
\geq \limsup_{m \to \infty} \sum_{j = 1}^{a_{m+1}} \frac{1}{j}
\geq \limsup_{m \to \infty} \ln(a_{m+1}).
\end{align*}
Since $A_{\alpha}(\theta) > 0$, the continued fraction entries of $\theta$ are unbounded and so this latter value is infinite.  Combining this with \eqref{eq:unltrametric_shift_words} and \eqref{eq:xyL} gives the required result.
\end{proof}

\begin{proof}[Proof of \Cref{prop:curious}~\ref{prop:curious_3}\ref{prop:curious_3a}]
Using \Cref{rmk:exp_q_k} and the assumption that $t > 1$, we conclude the following chain of inequalities.
\begin{align*}
\limsup_{m \to \infty} q_{m}^{t  r}  \sum_{k = m}^{\infty}  \sum_{j = 1}^{a_{k+1}} \frac{1}{(j  q_{k} + q_{k-1} - \mathds{1}_{2\mathbb{Z}}(k-m))^{t}}
&\leq \limsup_{m \to \infty} q_{m}^{t  r}  \sum_{k = m}^{\infty}  \frac{1}{q_{k}^{t}}  \sum_{j = 1}^{a_{k+1}} \frac{1}{j^{t}}\\
&\leq \limsup_{m \to \infty} q_{m}^{t  r}  \sum_{k = m}^{\infty}  \frac{t}{(t-1)  q_{k}^{t}}&\\
&\leq \limsup_{m \to \infty} \frac{t}{t - 1}  q_{m}^{t  r}  \sum_{k = m}^{\infty}  \frac{1}{q_{k}^{t}}\\
&\leq \limsup_{m \to \infty} \frac{t}{t - 1}  q_{m}^{t  r}  \sum_{j = 0}^{\infty}  \frac{(2 \sqrt{5})^{t}}{q_{m}^{t}\gamma^{j  t}}
= \limsup_{m \to \infty} \frac{t (2  \sqrt{5})^{t}  q_{m}^{t  (r - 1)}}{(t - 1)(1- \gamma^{-t})}
\end{align*}
For $r \in (0, 1)$ we observe that this latter value is zero and for $r = 1$ that it is finite.  This in tandem with \eqref{eq:unltrametric_shift_words} and \eqref{eq:xyL} yields that $\displaystyle \sup_{z \in \{ x, y \}} \limsup_{n \to \infty} \psi_{z, n}(1) < \infty$ and $\displaystyle \sup_{z \in \{ x, y \}} \limsup_{n \to \infty} \psi_{z, n}(r) = 0$, for $r \in (0, 1)$.
\end{proof}

\begin{proof}[Proof of \Cref{prop:curious}~\ref{prop:curious_3}\ref{prop:curious_3b}]
Observe that
\begin{align*}
\limsup_{m \to \infty} q_{m}^{rt}  \sum_{j = 1}^{a_{m+1}} \frac{1}{(j  q_{m})^{t}}
\geq \limsup_{m \to \infty} q_{m}^{t(r - 1)} \sum_{j = 1}^{a_{m+1}} \frac{1}{j^{t}}
&\geq \limsup_{m \to \infty} q_{m}^{t(r - 1)}  \frac{1 - (a_{m+1} + 1)^{1 - t}}{t-1}\\
&\geq \frac{1 - 2^{1-t}}{t-1} \limsup_{m \to \infty} q_{n}^{t(r-1)}.
\end{align*}
This with \eqref{eq:unltrametric_shift_words} and \eqref{eq:xyL} yields \!\! $\displaystyle \sup_{z \in \{ x, y \}} \limsup_{n \to \infty} \psi_{z, n}(1) > 0$ and \!\!$\displaystyle \sup_{z \in \{ x, y \}} \limsup_{n \to \infty} \psi_{z, n}(r) = \infty$, for $r > 1$.
\end{proof}

For our next proposition we require the following notation. As above let $X$ denote a Sturmian subshift of slope $\theta = [0; a_{1} + 1, a_{2}, \dots ]$ and let $x, y \in X$ denote the unique infinite words with $x\lvert_{\lvert \mathcal{R}_{n} \rvert} = \mathcal{R}_{n}$ and $y\lvert_{\lvert \mathcal{L}_{n} \rvert} = \mathcal{L}_{n}$, for all $n \in \mathbb{N}$.  By \Cref{prop:blackboard}, we have
\begin{align}\label{eq:unltrametric_shift_words_2}
\begin{aligned}
&x\lvert_{\lvert \mathcal{R}_{n} \rvert + j \lvert \mathcal{L}_{n+1} + 1}
=\mathcal{R}_{n} \underbrace{\mathcal{L}_{n + 1} \dots \mathcal{L}_{n+1}}_{j}(1),\\ 
&y\lvert_{\lvert \mathcal{L}_{n} \rvert + i \lvert \mathcal{R}_{n} \rvert + 1} 
= \mathcal{L}_{n} \underbrace{\mathcal{R}_{n} \dots \mathcal{R}_{n}}_{i}(0),\\
&\sigma^{(a_{2(n+1)} -j +1) \lvert \mathcal{L}_{n + 1} \rvert}(y)\lvert_{\lvert \mathcal{R}_{n} \rvert + j \lvert \mathcal{L}_{n+1} + 1}
 = \mathcal{R}_{n} \underbrace{\mathcal{L}_{n + 1} \dots \mathcal{L}_{n+1}}_{j}(0),\\
&\sigma^{(a_{2(n+1)-1} -i +1) \lvert \mathcal{R}_{n} \rvert}(x)\lvert_{\lvert \mathcal{L}_{n} \rvert + i \lvert \mathcal{R}_{n} \rvert + 1}
= \mathcal{L}_{n} \underbrace{\mathcal{R}_{n} \dots \mathcal{R}_{n}}_{i}(1).
\end{aligned}
\end{align}
for all $n \in \mathbb{N}$, $j \in \{ 1, 2, \dots, a_{2(n+1)} \}$ and $i \in \{ 1, 2, \dots, a_{2(n+1)-1} \}$.

The words $\sigma^{(a_{2(n+1)} + 1) \lvert \mathcal{L}_{n + 1} \rvert}(y)$ and $\sigma^{\lvert \mathcal{L}_{n} \rvert}(y)$ are distinct and as we will shortly see, although the ultra metric distance between these words and $x$ are equal, the respective spectral distances are not equal; the same holds for  $\sigma^{(a_{2(n+1)-1} + 1) \lvert \mathcal{R}_{n} \rvert}(x)$ and $\sigma^{\lvert \mathcal{R}_{n} \rvert}(x)$ and their ultra metric distance, respectively their spectral distance, to $y$.

For $n \in \mathbb{N}$, $j \in \{ 1, 2, \dots, a_{2(n+1)} \}$ and $i \in \{ 1, 2, \dots, a_{2(n+1)-1} \}$, we have that
\begin{align}\label{eq:unltrametric_shift-2}
d_{\delta}(x, \sigma^{(a_{2(n+1)} -j +1) \lvert \mathcal{L}_{n + 1} \rvert}(y)) = \delta_{j q_{2(n+1)-1} + q_{2n}}
\quad \text{and} \quad
d_{\delta}(\sigma^{(a_{2(n+1)-1} -i +1) \lvert \mathcal{R}_{n} \rvert}(x), y) = \delta_{i q_{2n} + q_{2n-1}},
\end{align}
and combining \Cref{cor:branching_points} and \eqref{eq:spectral_metric_bound}, we obtain that
\begin{align}\label{eq:xyL-2}
\begin{aligned}
d_{s, \delta}(x, &\sigma^{(a_{2(n+1)} -j +1) \lvert \mathcal{L}_{n + 1} \rvert}(y))\\
&= \sum_{l \geq j}^{a_{2(n +1)}} \delta_{l q_{2(n+1)-1} + q_{2n}} + \hspace{-0.5em} \sum_{k = 2(n + 1)}^{\infty} \sum_{l = 1}^{a_{k+1}} \delta_{l q_{k} + q_{k - 1} - \mathds{1}_{2\mathbb{Z}}(k-2(n+1)) (a_{2(n+1)} - j + 1) q_{2(n+1)-1}}\\
d_{s, \delta}(y, &\sigma^{(a_{2(n+1)-1} -i +1) \lvert \mathcal{R}_{n} \rvert}(x))\\
&= \sum_{l \geq i}^{a_{2(n+1)-1}} \delta_{l q_{2n} + q_{2n-1}} + \hspace{-0.75em}\sum_{k = 2(n+1)-1}^{\infty} \sum_{l = 1}^{a_{k+1}} \delta_{l q_{k} + q_{k - 1} - \mathds{1}_{2\mathbb{Z}}(k-(2(n+1)-1)) (a_{2(n+ 1)-1} - i + 1) q_{2n}}.
\end{aligned}
\end{align}
For $n \in \mathbb{N}$, $j \in \{ 1, 2, \dots, a_{2(n+1)} \}$, $i \in \{ 1, 2, \dots, a_{2n+1} \}$ and $r > 0$ set
\begin{align*}
\psi_{x, n}^{(j)}(r) \coloneqq \frac{d_{s, \delta}(x, \sigma^{(a_{2(n+1)} -j +1) \lvert \mathcal{L}_{n + 1} \rvert}(y))}{d_{\delta}(x, \sigma^{(a_{2(n+1)} -j +1) \lvert \mathcal{L}_{n + 1} \rvert}(y))^{r}}
\quad \text{and} \quad
\psi_{y, n}^{(i)}(r) \coloneqq \frac{d_{s, \delta}(\sigma^{(a_{2(n+1)-1} -i +1) \lvert \mathcal{R}_{n} \rvert}(x), y)}{d_{\delta}(\sigma^{(a_{2(n+1)-1} -i +1) \lvert \mathcal{R}_{n} \rvert}(x), y)^{r}}.
\end{align*}
 Notice that $\displaystyle{\limsup_{n \to \infty} \psi_{z, n}^{(j)}(r) \leq \psi_{z}(r)  \leq \psi(r)}$ for $z\in\{x,y\}$.

\begin{proposition}\label{prop:curious-2}
Let $\alpha > 1$, let $t > 1 - 1/\alpha$ and set $\delta \coloneqq (n^{-t})_{n \in \mathbb{N}}$.
\begin{enumerate}[label=(\arabic*),leftmargin=2.5em]
\item\label{prop:curious-2a} If $A_{\alpha}(\theta) < \infty$, then
\begin{align*}
\quad \sup_{z \in \{ x, y \}} \limsup_{n \to \infty} \; \sup \; \left\{ \psi_{z, n}^{(j)}(r) \colon  \; j \in \{ 1, \dots, a_{2(n+1)-\mathds{1}_{y}(z)} \} \right\}
\begin{cases}
= 0 &\!\!\text{if} \; 0 < r < 1 - (\alpha - 1)/(\alpha  t),\\ 
< \infty &\!\!\text{if} \; r = 1 - (\alpha - 1)/(\alpha  t).
\end{cases}
\end{align*}
\item\label{prop:curious-2b} If $A_{\alpha}(\theta) > 0$, then
\begin{align*}
\quad\sup_{z \in \{ x, y \}} \limsup_{n \to \infty} \; \sup \; \left\{ \psi_{z, n}^{(j)}(r) \colon  \; j \in \{ 1, \dots, a_{2(n+1)-\mathds{1}_{y}(z)} \} \right\}
\begin{cases}
= \infty &\!\!\text{if} \; r > 1 - (\alpha - 1)/(\alpha  t),\\ 
> 0 &\!\!\text{if} \; r = 1 - (\alpha - 1)/(\alpha  t).
\end{cases}
\end{align*}
\end{enumerate}
\end{proposition}

We divide the proof of each part of the above proposition into three cases: the first case when $t \in (1 - 1/\alpha, 1)$, the second case when $t = 1$ and the third case when $t > 1$.  We will also use the following lemma and remark in the proof of \Cref{prop:curious-2}.

\begin{lemma}\label{lem:infinite_median_sum_t_1}
Let $\alpha > 1$ and let $t > 1 - 1/\alpha$.  Let $X$ denote a Sturmian subshift of slope $\theta \in [0, 1/2]$ where $A_{\alpha}(\theta) < \infty$. Given $r \in (0,  \min \{ 1, \alpha - (\alpha - 1)/t \})$ and given $\epsilon > 0$, there exists $M = M_{t, r} \in \mathbb{N}$ such that for all $m \geq M$ and $j \in \{ 1, 2, \dots, a_{m+2} \}$,
\begin{align*}
0 < 
(j q_{m+1} + q_{m})^{t r} 
\sum_{k = m + 2}^{\infty} \sum_{l = 1}^{a_{k+1}} \frac{1}{(l q_{k} + q_{k - 1} - \mathds{1}_{2\mathbb{Z}}(k-(m+2)) (a_{m+2} - j + 1) q_{m+1})^{t}} < \epsilon.
\end{align*}
\end{lemma}

\begin{proof}[Proof of \Cref{lem:infinite_median_sum_t_1}]
The lower bound follows trivial since the quantities involved are non-negative.  Since $A_{\alpha}(\theta) < \infty$ there exists a constant $c > 1$ so that $a_{m+1} \leq c  q_{m}^{\alpha - 1}$, for all $m \in \mathbb{N}$, and hence, for all $j \in \{ 1, 2, \dots, a_{m+2} \}$, we have the following chain of inequalities, where $I_{\alpha}$ denotes the interval $(-1/\alpha, 1)$.
\begin{align*}
& (j q_{m+1} + q_{m})^{t r}
\sum_{k = m + 2}^{\infty} \sum_{l = 1}^{a_{k+1}} \frac{1}{(l q_{k} + q_{k - 1} - \mathds{1}_{2\mathbb{Z}}(k - (m+2)) (a_{m+2} - j + 1) q_{m+1})^{t}}\\
& \leq \begin{cases}
\displaystyle \frac{1+2^{t}}{1 - t}  q_{m+2}^{t  r} \!\! \sum_{k = m+2}^{\infty}
\frac{(a_{k+1} q_{k} + q_{k - 1} - \mathds{1}_{2\mathbb{Z}}(k - (m+2)) (a_{m+2} - j + 1) q_{m+1})^{1-t}}{q_{k}}
&\hspace{-0.275em}\text{if} \; t \in I_{\alpha}\\[1em]
\parbox{37em}{$\displaystyle 3 q_{m+2}^{r} \!\! \sum_{k = m+2}^{\infty} \frac{\ln(a_{k+1} q_{k} + q_{k - 1} - \mathds{1}_{2\mathbb{Z}}(k - (m+2)) (a_{m+2} - j + 1) q_{m+1})}{q_{k}}$} &\hspace{-0.275em}\text{if} \;  t = 1\\[1em]
\displaystyle
\frac{1 + 2^{t}}{t-1}  (j q_{m+1} + q_{m})^{t r} \!\! \sum_{k = m+2}^{\infty}
\frac{(q_{k} + q_{k - 1} - \mathds{1}_{2\mathbb{Z}}(k - (m+2)) (a_{m+2} - j + 1) q_{m+1})^{1-t}}{q_{k}}
&\hspace{-0.275em}\text{if} \; t > 1
\end{cases}\\
& \leq \begin{cases}
\displaystyle \frac{1 + 2^{t}}{1 - t}  q_{m+2}^{t  r} \sum_{k = m+2}^{\infty}
\frac{(a_{k+1} q_{k} + q_{k - 1} )^{1-t}}{q_{k}}
&\text{if} \; t \in I_{\alpha}\\[1em]
\parbox{37em}{$\displaystyle 3 q_{m+2}^{t  r} \sum_{k = m+2}^{\infty} \frac{\ln(a_{k+1} q_{k} + q_{k - 1})}{q_{k}}$} &\text{if} \;  t = 1\\[1em]
\displaystyle
\frac{1 + 2^{t}}{t-1} \left( \frac{(j q_{m+1} + q_{m})^{1-t+tr}}{q_{m+2}} + 
q_{m+2}^{t r} \sum_{k = m+3}^{\infty}
\frac{q_{k}^{1-t}}{q_{k}} 
\right)
&\text{if} \; t > 1
\end{cases}\\
& \leq \begin{cases}
\displaystyle \frac{1 + 2^{t}}{1 - t}  c^{1-t} 2^{1-t}  q_{m+2}^{t  r} \sum_{k = m+2}^{\infty} \frac{1}{q_{k}^{\alpha  t - (\alpha - 1)}} &\text{if} \; t \in I_{\alpha}\\[1em]
\parbox{37em}{$\displaystyle 3 q_{m+2}^{r} \sum_{k = m+2}^{\infty} \frac{\ln(2c) + \alpha \ln(q_{k})}{q_{k}}$} &\text{if} \;  t = 1\\[1em]
\displaystyle
\frac{1 + 2^{t}}{t-1} \left( (j q_{m+1} + q_{m})^{t(r-1)} +
 q_{m+2}^{t  r} \sum_{k = m+3}^{\infty} \frac{1}{q_{k}^{t}} \right)  
&\text{if} \; t > 1
\end{cases}\\
& \leq \begin{cases}
\displaystyle \frac{(1 + 2^{t})  c^{1-t}  2^{1-t}  2  \sqrt{5} }{1 - t}  q_{m+2}^{t  r - \alpha  t + (\alpha - 1)} \sum_{i = 0}^{\infty} \frac{1}{\gamma^{i  (\alpha  t - (\alpha - 1))}} &\text{if} \; t \in I_{\alpha}\\[1em]
\parbox{37em}{$\displaystyle 6  \sqrt{5}  q_{m+2}^{r - 1} \ln(q_{m+2}) \sum_{i = 0}^{\infty} 
\frac{\ln(2c) + 1 + \alpha i  \ln(\gamma)}{\gamma^{i }}$} &\text{if} \;  t = 1\\[1em]
\displaystyle \frac{1 + 2^{t}}{t - 1}  \left( (q_{m+1} + q_{m})^{t(r-1)}  + 2 \sqrt{5} q_{m+2}^{t  (r - 1)} \sum_{i = 1}^{\infty} \frac{1}{\gamma^{i  t}} \right)&\text{if} \; t > 1\\
\end{cases}
\end{align*}
In the last inequality we have used the result given in \Cref{rmk:exp_q_k} and the fact, for $x > \mathrm{e}^{1}$, that the function $x \mapsto \ln(x)/x$ is strictly decreasing.  Since $r \in (0, \min\{1, \alpha - (\alpha - 1)/t \})$, $\gamma > 1$, $t > 1 - 1/\alpha$ and the sequence $(q_{n})_{n \in \mathbb{N}}$ is unbounded and monotonically increasing, the result follows.
\end{proof}

Given $m \in \mathbb{N}$, $j \in \{ 1, 2, \dots, a_{m+2}\}$, $r > 0$ and $t > 0$ set
\begin{align}\label{eq:phi-1-2-3}
\phi(m, j, r, t) \coloneqq (jq_{m+1} + q_{m})^{t r} 
\sum_{l = j}^{a_{m + 2}} \frac{1}{(l q_{m+1} + q_{m})^{t}}.
\end{align}
By \eqref{eq:unltrametric_shift-2}, \eqref{eq:xyL-2} and \Cref{lem:infinite_median_sum_t_1}, to prove \Cref{prop:curious-2}, it is sufficient to show, if $A_{\alpha}(\theta) < \infty$, then
\begin{align*}
\limsup_{m \to \infty} \sup \left\{ \phi(m, j, r, t) \colon j \in \{ 1, 2, \dots, a_{m+2} \} \right\} &\begin{cases}
= 0 & \text{if} \; 0 < r < 1 - (\alpha - 1)/(\alpha  t),\\ 
< \infty & \text{if} \; r = 1 - (\alpha - 1)/(\alpha  t),
\end{cases}
\intertext{and if $A_{\alpha}(\theta) > 0$, then}
\limsup_{m \to \infty} \sup \left\{ \phi(m, j, r, t) \colon j \in \{ 1, 2, \dots, a_{m+2} \} \right\} &\begin{cases}
= \infty & \text{if} \; r > 1 - (\alpha - 1)/(\alpha  t),\\ 
> 0 & \text{if} \; r = 1 - (\alpha - 1)/(\alpha  t).
\end{cases}
\end{align*}

\begin{proof}[Proof of \Cref{prop:curious-2}~\ref{prop:curious-2a}]
Case $t \in (1 - 1/\alpha, 1)$: Since $A_{\alpha}(\theta) < \infty$, there exists a constant $c > 1$ so that $a_{m+1} \leq c q_{m}^{\alpha - 1}$, for all $m \in \mathbb{N}$.  With this at hand, for $0 < r \leq 1 - (\alpha - 1)/(\alpha t)$, we may deduce the following chain of inequalities.
\begin{align}\label{eq:copy-paste-argument-2}
\begin{aligned}
\limsup_{m \to \infty} \!\!\sup_{1 \leq j \leq a_{m+2}} \!\!\phi(m, j, r, t)
&\leq \limsup_{m \to \infty} q_{m+2}^{rt} \sum_{l = 1}^{a_{m + 2}} \frac{1}{(l q_{m+1}+ q_{m})^{t}}\\
&\leq \limsup_{m \to \infty} \frac{1+2^{t}}{1-t} q_{m+2}^{rt} \frac{(a_{m+1}q_{m+1}+q_{m})^{1-t}}{q_{m+1}}\\
&\leq \limsup_{m \to \infty} \frac{1+2^{t}}{1-t} q_{m+2}^{rt}  \frac{q_{m+2}^{1-t}}{q_{m+1}}\\
&\leq \limsup_{m \to \infty} \frac{1+2^{t}}{1-t} (2c)^{1- t(1-r)} q_{m+1}^{\alpha- \alpha t(1-r) - 1}
\end{aligned}
\end{align}
This in tandem with the facts that $1 - (\alpha- 1)/(\alpha t) < \alpha - (\alpha - 1)/t$ if and only if $t > 1 - 1/\alpha$ and that $(q_{n})_{n \in \mathbb{N}}$ is an unbounded monotonic sequence, yields the result.

Case $t = 1$: Since $A_{\alpha}(\theta) < \infty$, there exists a constant $c > 1$ so that $a_{m+1} \leq c q_{m}^{\alpha - 1}$, for all $m \in \mathbb{N}$, and since, for $r > 0$, the function $x \mapsto x^{r} \left( \ln(a_{m + 2}) - \ln(x) \right)$, with domain $[0, \infty)$, is maximised at $x = a_{m+2} \mathrm{e}^{-1/r}$, we have, for $ 0 < r \leq  1 - (\alpha - 1)/(\alpha t) = 1/\alpha$, that
\begin{align*}
\limsup_{m \to \infty} \sup_{1 \leq j \leq a_{m+2}} \phi(m, j, r, t)
&\leq \limsup_{m \to \infty} \sup_{1 \leq j \leq a_{m+2}} q_{m}^{r - 1} + 2^{r} q_{m + 1}^{r - 1}j^{r} \sum_{l = j +1}^{a_{m+2}} \frac{1}{l}\\
&\leq \limsup_{m \to \infty} \sup_{1 \leq j \leq a_{m+2}} q_{m}^{r - 1} + 2^{r} q_{m + 1}^{r - 1}j^{r} (\ln(a_{m+2}) - \ln(j))\\
&\leq \limsup_{m \to \infty} q_{m}^{r - 1} + \frac{2^{r} \mathrm{e}^{-1}}{r} q_{m + 1}^{r - 1} a_{m+2}^{r}\\
&\leq \limsup_{m \to \infty} q_{m}^{r - 1} + \frac{2^{r} \mathrm{e}^{-1}c^{r}}{r} q_{m + 1}^{\alpha r - 1}.
\end{align*}
This in tandem with the fact that $(q_{n})_{n \in \mathbb{N}}$ is a monotonic unbounded sequence, yields the result.

Case $t >1$:
Since $A_{\alpha}(\theta) < \infty$, there is a constant $c > 1$ with $a_{m+1} \leq c q_{m}^{\alpha - 1}$, for all $m \in \mathbb{N}$.  Further, since $0 < r \leq 1- (\alpha - 1)/(\alpha t)$  and since $(q_{n})_{n \in \mathbb{N}}$ is an unbounded monotonic sequence, we may deduce the following chain of inequalities.
\begin{align*}
&\limsup_{m \to \infty} \sup_{1 \leq j \leq a_{m+2}} \phi(m, j, r, t)\\
&\leq \limsup_{m \to \infty} \sup_{1 \leq j \leq a_{m+2}} q_{m}^{t(r-1)} +  2^{tr} q_{m+1}^{t(r -1)} j^{t r} \sum_{l = j + 1}^{a_{m + 2}} \frac{1}{l^{t}}\\
&\leq \limsup_{m \to \infty} \sup_{1 \leq j \leq a_{m+2}} q_{m}^{t(r-1)} +  \frac{2^{tr}}{t - 1} q_{m+1}^{t(r -1)} j^{t(r-1)+1}\\\displaybreak[1]
&\leq \begin{cases}
\parbox{22.5em}{$\displaystyle \limsup_{m \to \infty} \; q_{m}^{t(r-1)} +  \frac{2^{tr}}{t - 1} q_{m+1}^{t(r -1)}$} & \text{if} \; r \leq 1 - 1/t\\[1em]
\displaystyle \limsup_{m \to \infty} \; q_{m}^{t(r-1)} +  \frac{2^{tr}}{t - 1} q_{m+1}^{t(r -1)} a_{m+2}^{t(r-1)+1} & \text{if} \; 1 - 1/t < r \leq 1 - (\alpha  -1)/(\alpha t)
\end{cases}\\
&\leq \begin{cases}
\displaystyle \limsup_{m \to \infty} \; q_{m}^{t(r-1)} +  \frac{2^{tr}}{t - 1} q_{m+1}^{t(r -1)} & \text{if} \; r \leq 1 - 1/t\\[1em]
\parbox{22.5em}{$\displaystyle \limsup_{m \to \infty} \; q_{m}^{t(r-1)} +  \frac{2^{tr} c^{t(r-1)+1}}{t - 1} q_{m+1}^{t(r -1)} q_{m+1}^{(\alpha-1)(t(r-1)+1)}$} & \text{if} \; 1 - 1/t < r \leq 1 - (\alpha  -1)/(\alpha t)
\end{cases}\\
&\leq \begin{cases}
\parbox{22.5em}{$\displaystyle \limsup_{m \to \infty} \; q_{m}^{t(r-1)} +  \frac{2^{tr}}{t - 1} q_{m+1}^{t(r -1)}$} & \text{if} \; r < 1 - 1/t\\[1em]
\displaystyle \limsup_{m \to \infty} \; q_{m}^{t(r-1)} +  \frac{2^{tr} c^{t(r-1)+1}}{t - 1} q_{m+1}^{(\alpha - 1) + \alpha t (r - 1)} & \text{if} \; 1 - 1/t < r < 1 - (\alpha  -1)/(\alpha t)\\[1em]
\displaystyle \limsup_{m \to \infty} \; q_{m}^{t(r-1)} +  \frac{2^{tr} c^{t(r-1)+1}}{t - 1} & \text{if} \; r = 1 - (\alpha  -1)/(\alpha t)
\end{cases}
\end{align*}
This in tandem with the fact that $(q_{n})_{n \in \mathbb{N}}$ is an unbounded monotonic sequence, yields the result.
\end{proof}

\begin{proof}[Proof of \Cref{prop:curious-2}~\ref{prop:curious-2b}]
Case $t\in(1-1/\alpha,1)$:
Since $A_{\alpha}(\theta) > 0$, there is an increasing sequence of integers $\{ n_{k} \}_{k \in \mathbb{N}}$ with $2  a_{n_{k}+2} > A_{\alpha}(\theta) q_{n_{k}+1}^{\alpha-1} > 18$.  Combing this with the fact that $(q_{n})_{n \in \mathbb{N}}$ is an unbounded monotonic sequence, and setting $j_{m} = \lceil a_{n_{m}+2}/2 \rceil$, we have that
\begin{align*}
\limsup_{m \to \infty} \; \phi(n_{m}, j_{m}, r, t)
&\geq \limsup_{m \to \infty} \; ( \lceil a_{n_{m}+2}/2 \rceil q_{n_{m}+1} + q_{n_{m}})^{t r} \frac{q_{n_{m}+2}^{1-t} - ( \lceil a_{n_{m}+2}/2 \rceil q_{n_{m}+1} + q_{n_{m}})^{1 - t}}{(1-t) q_{n_{m} +1}}\\
&\geq \limsup_{m \to \infty}  \frac{1 - (2/3)^{1-t}}{2^{tr}(1-t)} \frac{q_{n_{m}+2}^{1-t(1-r)}}{q_{n_{m}+1}}\\
&\geq \limsup_{m \to \infty}  \frac{(1 - (2/3)^{1-t})  A_{\alpha}(\theta)^{1 - t(1-r)}}{2^{2tr + 1-t} (1-t)} q_{n_{m}+1}^{\alpha-\alpha t(1-r)-1}.
\end{align*}
This in tandem with the fact that $(q_{n})_{n \in \mathbb{N}}$ is an unbounded monotonic sequence, yields the results.

Case $t =1$: Since $A_{\alpha}(\theta) > 0$, there exists an increasing sequence of non-negative integers $\{ n_{k} \}_{k \in \mathbb{N}}$ so that $2 a_{n_{k}+2} > A_{\alpha}(\theta) q_{n_{k}+1}^{\alpha-1} > 18$.  Setting $j_{m} = \lceil a_{n_{m}+2}/2 \rceil$, we have that
\begin{align*}
\limsup_{m \to \infty} \phi(m, j_{m}, r, t) 
\geq& \limsup_{m \to \infty} \; \frac{(j_{n_{m}}  q_{n_{m}+1} + q_{n_{m}})^{r}}{q_{n_{m}+1}} \left( \ln(q_{n_{m} + 2}) - \ln(j_{n_{m}}  q_{n_{m} + 1} + q_{n_{m}}) \right)\\
\geq& \limsup_{m \to \infty} \; \frac{1}{2^r} \frac{q_{n_{m}+2}^{r}}{q_{n_{m}+1}} \left( \ln(q_{n_{m} + 2}) - \ln \left(\frac{2  q_{n_{m} + 2}}{3} \right) \right)\\
\geq& \limsup_{m \to \infty} \; 2^{-r} \ln(3/2)  a_{n_{m}+2}^{r} q_{n_{m}+1}^{r-1}\\
\geq& \limsup_{m \to \infty} \; 2^{-2r} \ln(3/2)  A_{\alpha}(\theta)^{r} q_{n_{m}+1}^{r \alpha-1}.
\end{align*}
This in tandem with the fact that $(q_{n})_{n \in \mathbb{N}}$ is an unbounded monotonic sequence, yields the results.

Case $t >1$: Since $A_{\alpha}(\theta) > 0$, there exists an increasing sequence of natural numbers $\{ n_{k} \}_{k \in \mathbb{N}}$ so that $2  a_{n_{k}+2} > A_{\alpha}(\theta) q_{n_{k}+1}^{\alpha-1} > 18$.  Setting $j_{m} = \lceil a_{n_{m}+2}/2 \rceil$, we have that
\begin{align*}
\limsup_{m \to \infty} \; \phi(n_{m}, j_{m}, r, t)
&\geq \limsup_{m \to \infty} \; ( \lceil a_{n_{m}+2}/2 \rceil q_{n_{m}+1} + q_{n_{m}})^{t r} \frac{( \lceil a_{n_{m}+2}/2 \rceil q_{n_{m}+1} + q_{n_{m}})^{1 - t} - q_{n_{m}+2}^{1-t}}{(t - 1) q_{n_{m} +1}}&\\
&\geq \limsup_{m \to \infty}  \frac{(2/3)^{1-t} - 1}{2^{tr}(t-1)} \frac{q_{n_{m}+2}^{1-t(1-r)}}{q_{n_{m}+1}}\\
&\geq \limsup_{m \to \infty}  \frac{((2/3)^{1-t} - 1)  A_{\alpha}(\theta)^{1 - t(1-r)}}{2^{2tr + 1 - t} (1-t)} q_{n_{m}+1}^{\alpha-\alpha t(1-r)-1}.
\end{align*}
This in tandem with the fact that $(q_{n})_{n \in \mathbb{N}}$ is an unbounded monotonic sequence, yields the results.
\end{proof}

\begin{proposition}\label{prop:lower_Holder_reg}
Let $\alpha > 1$ and let $X$ denote a Sturmian subshift of slope $\theta \in [0, 1/2]$.  Let $t > 1 - 1/\alpha$, set $\delta = (\delta_{n})_{n \in \mathbb{N}}$ with $\delta_{n} = n^{-t}$.  If $A_{\alpha}(\theta) < \infty$, that is there exists $c > 0$ so that $a_{m+1} \leq c  q_{m}^{\alpha - 1}$, for all $m \in \mathbb{N}$, then, for $r > 0$,
\begin{align*} 
\psi_{w}(r) 
\leq 2 (c + 2)^{tr} \sup_{z \in \{ x, y \}} \limsup_{n \to \infty} \; \sup \; \bigg\{ \psi_{z, n}^{(k)}(r) \colon k \in \{ 1, \dots, a_{2(n+1)-\mathds{1}_{y}(z)} \} \bigg\} \cup \bigg\{ \psi_{z, n}(\alpha r) \bigg\}.
\end{align*}
\end{proposition}

\begin{proof}
Let $w = (w_{1}, w_{2}, \dots ) \in X$ be fixed and let $n \geq 2$ denote a natural number with $\overline{b}_{n}(w) = 1$.  Set $k_{z}(n) = \sup \{ l \in \{ 1, 2, \dots n \} \colon \overline{b}_{l}(z) = 1 \}$, where $z \in \{ x, y\}$. (Note that $k_{x}(m(n)) \neq k_{y}(m(n))$ as there exists a unique right special word per length.)  By definition we have $\overline{b}_{k_{z}(n)}(z)=1$,  and so, by \Cref{cor:branching_points}, there exist $l(n), l'(n) \in \mathbb{N}$, $p(n) \in \{ 1, 2, \dots, a_{2(l(n) + 1)}  \}$ and $p'(n) \in \{ 1, 2, \dots, a_{2(l(n) + 1) - 1}  \}$, with
\begin{align*}
x\lvert_{k_{x}(n)} = \mathcal{R}_{l(n)} \underbrace{\mathcal{L}_{l(n)+1} \dots \mathcal{L}_{l(n)+1}}_{p(n)} \quad \text{and} \quad y\lvert_{k_{y}(n)} = \mathcal{L}_{l'(n)} \underbrace{\mathcal{R}_{l'(n)} \dots \mathcal{R}_{l'(n)}}_{p'(n)}.
\end{align*}
An application of \Cref{rmk:BFMS-2002} and \Cref{prop:blackboard}, yields that
\begin{align}\label{eq:n-m}
\inf\{ l \in \mathbb{N} \colon \overline{b}_{n + l}(w) = 1 \} \geq \begin{cases}
\lvert \mathcal{L}_{l(n)+1}\rvert = q_{2l(n)+1} & \text{if} \; w_{n+1} = 1, 
\; \text{and} \; p(n) \neq a_{2(l+1)},\\
\lvert \mathcal{R}_{l'(n)}\rvert = q_{2l'(n)} & \text{if} \; w_{n+1} = 0,  
\; \text{and} \; p'(n) \neq a_{2(l'(n)+1) - 1},\\
\lvert \mathcal{L}_{l(n)+2}\rvert = q_{2l(n)+3} & \text{if} \; w_{n+1} = 1 \; \text{and} \; p = a_{2(l(n)+1)},\\
\lvert \mathcal{R}_{l'(n) + 1}\rvert = q_{2(l'(n)+1)} & \text{if} \; w_{n+1} = 0 \; \text{and} \; p'(n) = a_{2(l'(n)+1) - 1}.
\end{cases}
\end{align}
Thus, since $\delta_{k} = k^{-t}$, we have that
\begin{align}\label{eq:sum_key}
\sum_{k \geq n} \overline{b}_{k}(w) \delta_{k} \leq \sum_{k \geq k_{x}(n)} \overline{b}_{k}(x) \delta_{k} + \sum_{k \geq k_{y}(n)} \overline{b}_{k}(y) \delta_{k}.
\end{align}
Let $w \in X$ be fixed.   We set $m (0) \coloneqq 0$, define $m(n)\coloneqq \min\{k>m(n-1)\colon \overline{b}_{m(n)}(w) = 1\}$ and let $(w^{(n)})_{n \in \mathbb{N}}$ denote a sequence in $X$ such that $w^{(n)}\lvert_{m(n)} = w\lvert_{m(n)}$ and $w^{(n)}\lvert_{m(n)+1} \neq w\lvert_{m(n)+1}$. Combining the above with \eqref{eq:spectral_metric_bound}, \eqref{eq:unltrametric_shift_words}, \eqref{eq:unltrametric_shift_words_2} and \eqref{eq:sum_key} we conclude the following.
\begin{align*}
d_{s, \delta}\left(w, w^{(n)}\right)
&\leq 2 \left( \sum_{k \geq k_{x}(m(n))} \overline{b}_{k}(x) \delta_{k} + \sum_{k \geq k_{y}(m(n))} \overline{b}_{k}(y) \delta_{k} \right)\\
&\leq \begin{cases}
\parbox{22em}{$\displaystyle 2 d_{s, \delta}(x, \sigma^{(a_{2(l(m(n))+1)} -p(m(n)) +1) \lvert \mathcal{L}_{l(m(n)) + 1} \rvert}(y))$}
& \parbox{15em}{$\text{if} \; k_{x}(m(n)) <  k_{y}(m(n)) \; \text{and} \\ p(m(n)) \neq a_{2(l(m(n))+1)}$}\\[1em]
2 d_{s, \delta}(\sigma^{(a_{2l'(m(n))+1} -p'(m(n)) +1) \lvert \mathcal{R}_{l'(m(n))} \rvert}(x), y)
& \parbox{15em}{$\text{if} \; k_{x}(m(n)) >  k_{y}(m(n)) \; \text{and} \\ p'(m(n)) \neq a_{2(l'(m(n))+1) - 1}$}\\[1em]
2 d_{s, \delta}(x, \sigma^{\lvert \mathcal{L}_{l(m(n))+1} \rvert}(y) )
& \parbox{15em}{$\text{if} \; k_{x}(m(n)) <  k_{y}(m(n))\; \text{and} \\ p(m(n)) = a_{2(l(m(n))+1)}$}\\[1em]
2 d_{s, \delta}(\sigma^{\lvert \mathcal{R}_{l'(m(n))}\rvert}(x) , y )
& \parbox{15em}{$\text{if} \; k_{x}(m(n)) >  k_{y}(m(n)) \; \text{and} \\ p'(m(n)) = a_{2(l'(m(n))+1) - 1}$}
\end{cases}
\intertext{On the other hand by \eqref{eq:unltrametric_shift_words}, \eqref{eq:unltrametric_shift-2} and \eqref{eq:n-m}, for $r \in (0, 1)$, we have that}
d_{\delta}(w, w^{(n)})^{-r} &= \delta_{m(n)}^{-r}
= (m(n))^{rt}\\\displaybreak[1]
&\leq 
\begin{cases}
\parbox{22em}{$\displaystyle 2^{rt} \lvert \mathcal{R}_{l(m(n))} \underbrace{\mathcal{L}_{l(m(n))+1} \dots \mathcal{L}_{l(m(n))+1}}_{p(m(n))} \rvert^{rt}$}
& \parbox{15em}{$\text{if} \; k_{x}(m(n)) <  k_{y}(m(n)) \; \text{and} \\ p(m(n)) \neq a_{2(l(m(n))+1)},$}\\
2^{rt} \lvert \mathcal{L}_{l'(m(n))} \underbrace{\mathcal{R}_{l'(m(n))} \dots \mathcal{R}_{l'(m(n))}}_{p'(m(n))} \rvert^{rt} 
& \raisebox{-0.5em}{\parbox{15em}{$\text{if} \; k_{x}(m(n)) >  k_{y}(m(n)) \; \text{and} \\ p'(m(n)) \neq a_{2(l'(m(n))+1) - 1},$}}\\
\lvert \mathcal{R}_{l(m(n))+1} \mathcal{L}_{l(m(n))+2}\rvert^{rt} 
& \parbox{15em}{$\text{if} \; k_{x}(m(n)) <  k_{y}(m(n))\; \text{and} \\ p(m(n)) = a_{2(l(m(n))+1)},$}\\[1em]
\lvert \mathcal{L}_{l'(m(n))+1} \mathcal{R}_{l'(m(n))+1}\rvert^{rt}
& \parbox{15em}{$\text{if} \; k_{x}(m(n)) >  k_{y}(m(n)) \; \text{and} \\ p'(m(n)) = a_{2(l'(m(n))+1) - 1},$}
\end{cases}\\[0.5em]
&\leq
\begin{cases}
\parbox{22em}{$\displaystyle 2^{rt} d_{\delta}(x, \sigma^{(a_{2(l(m(n))+1)} -p(m(n)) +1) \lvert \mathcal{L}_{l(m(n)) + 1} \rvert}(y))^{-r}$}
& \parbox{15em}{$\text{if} \; k_{x}(m(n)) <  k_{y}(m(n)) \; \text{and} \\ p(m(n)) \neq a_{2(l(m(n))+1)},$}\\[1em]
2^{rt} d_{\delta}(\sigma^{(a_{2l'(m(n))+1} -p'(m(n)) +1) \lvert \mathcal{R}_{l'(m(n))} \rvert}(x), y)^{-r}
& \parbox{15em}{$\text{if} \; k_{x}(m(n)) >  k_{y}(m(n)) \; \text{and} \\ p'(m(n)) \neq a_{2(l'(m(n))+1) - 1},$}\\[1em]
(c + 2)^{tr} d_{\delta}(x, \sigma^{\lvert \mathcal{L}_{l(m(n)) + 1} \rvert}(y))^{-\alpha r}
& \parbox{15em}{$\text{if} \; k_{x}(m(n)) <  k_{y}(m(n))\; \text{and}\\ p(m(n)) = a_{2(l(m(n))+1)},$}\\[1em]
(c + 2)^{tr}d_{\delta}(\sigma^{\lvert \mathcal{R}_{l'(m(n))} \rvert}(x), y)^{-\alpha r}
& \parbox{15em}{$\text{if} \; k_{x}(m(n)) >  k_{y}(m(n)) \; \text{and} \\ p'(m(n)) = a_{2(l'(m(n))+1) - 1}$.}
\end{cases}
\end{align*}
 To complete the proof we observe the following.  Since $d_\delta$ induces the discrete product topology on $X$, any sequence in $X \setminus \{\ w \}$ converging to $w$ with respect to $d_\delta$ is a subsequence of a sequence of the form $(w^{(n)})_{n\in\mathbb{N}}$, and hence
\begin{align*}
\psi_{w}(r) = \limsup_{n \to \infty} \sup \left\{\frac{d_{s, \delta}(w, v)}{d_{\delta}(w, v)^{r}} \colon 
v \in X,
v\lvert_{m(n)} = w\lvert_{m(n)}  \; \text{and} \;  
v\lvert_{m(n)+1} \neq w\lvert_{m(n)+1} 
\right\}.
\end{align*}
This completes the proof.
\end{proof}

\section{Proofs}

\subsection{Proof of \Cref{prop:1-repulsive=repulsive,prop:Not_Metric,prop:liminf_psi_r}}\label{sec:1-repulsive=repulsive}

\begin{proof}[Proof of \Cref{prop:1-repulsive=repulsive}]
Let $X$ be a repulsive Sturmian subshift of slope $\theta = [0; a_{1}+1, a_{2}, \dots]$ and observe that $0 < \ell \leq \ell_{1}$.  Recall, repulsive implies that the continued fraction entries of $\theta$ are bounded.  Suppose that $a_{k} \neq 1$ infinitely often. By \Cref{prop:blackboard}, for all $k \in \mathbb{N}$, we have
\begin{align}\label{eq:WwW'w'}
W \coloneqq \underbrace{\mathcal{L}_{k} \dots \mathcal{L}_{k}}_{a_{2k}},
\quad
w \coloneqq \underbrace{\mathcal{L}_{k} \dots \mathcal{L}_{k}}_{a_{2k}-1},
\quad
W' \coloneqq \underbrace{\mathcal{R}_{k-1} \dots \mathcal{R}_{k-1}}_{a_{2k - 1}},
\quad \text{and} \quad
w' \coloneqq \underbrace{\mathcal{R}_{k-1} \dots \mathcal{R}_{k-1}}_{a_{2k - 1}-1}
\end{align}
all belong to $\mathcal{L}(X)$. Hence, letting $k_{n}$ denote the integers with $a_{k_{n}} \neq 1$, we have $A_{\alpha, a_{k_{n}}q_{k_{n}-1}} \leq (a_{k_{n}} - 1)^{-1}$, and so, $\ell_{1} = \liminf_{n \to \infty} A_{1, n} \leq \liminf_{n \to \infty} (a_{k_{n}} - 1)^{-1} \leq 1$.

Suppose we do not have that $a_{k} \neq 1$ infinitely often, namely that there exists $N \in \mathbb{N}$ such that $a_{N+j} = 1$, for all $j \in \mathbb{N}$.  In which case the continued fraction entries of $\theta$ are bounded and so the sequence $(q_{k+1}/q_{k})$ is convergent, with a non-zero and finite limit $L$.  Setting $W \coloneqq \mathcal{L}_{N+j+2} = \mathcal{L}_{N+j+1} \mathcal{R}_{N+j} \mathcal{L}_{N+j+1}$, $w \coloneqq \mathcal{L}_{N+j+1}$, $W' \coloneqq \mathcal{R}_{N+j+2} = \mathcal{R}_{N+j+1} \mathcal{L}_{N+j+1} \mathcal{R}_{N+j+1}$, $w' \coloneqq \mathcal{R}_{N+j+1}$, by \Cref{prop:blackboard}, we have that 
\begin{align*}
\ell_{1} = \liminf_{n \to \infty} A_{1, n} \leq 1 + \lim_{n \to \infty} \frac{q_{n+1}}{q_{n}} = 1 + L.
\end{align*}
For the forward implication, we show the contra-positive.  Recall that a Sturmian subshift $X$ of slope $\theta = [0; a_{1} + 1, a_{2}, \dots]$ is repulsive if and only if the continued fraction entries of $\theta$ are bounded.  Therefore, if $X$ is not repulsive the continued fraction entries of $\theta$ are unbounded.  Letting $W, w, W', w'$ be as in \eqref{eq:WwW'w'}, \Cref{prop:blackboard} implies, for all integers $k \geq 2$ with $a_{k} \neq 1$, that  $A_{\alpha, a_{k}q_{k-1}} \leq (a_{k} - 1)^{-1}$, and so, 
\begin{align*}
\ell_{1} = \liminf_{n \to \infty} A_{1, n} \leq \liminf_{n \to \infty} (a_{k_{n}} - 1)^{-1} = 0,
\end{align*}
yielding that $X$ is not $1$-repulsive.
\end{proof}

\begin{proof}[Proof of \Cref{prop:Not_Metric}]
For the first part of the result let $[0; a_{1} + 1, a_{2}, \dots ]\in [0,1/2]$ be the continued fraction expansion of $\theta$.  Since $\alpha \geq 1/(1 - t)$, $A_{\alpha}(\theta) \neq 0$ and since $(q_{k})_{k \in \mathbb{N}}$ is an unbounded monotonic sequence, there exists a sequence of natural numbers $(m_{i})_{i \in \mathbb{N}}$, so that 
\begin{align*}
0 < \min\left\{ 1, \frac{\textup{A}_{\alpha}(\theta)}{2} \right\} < a_{m_{i} + 1} q_{m_{i}}^{1 - \alpha} \leq a_{m_{i} + 1} q_{m_{i}}^{1 - 1/(1 - t)} \quad \text{and} \quad a_{m_{i}} \geq 4,
\end{align*}
and thus, for all $n \in \mathbb{N}$,
\begin{align*}
\sum_{k = n}^{\infty}  \sum_{j = 1}^{a_{k+1}}  \delta_{j q_{k} + q_{k - 1} - \mathds{1}_{2\mathbb{Z}}(k-n) q_{n-1}}
\geq \sum_{k = n}^{\infty}  \sum_{j = 1}^{a_{k + 1}}  \frac{1}{(j q_{k} + q_{k-1})^{t}}
&\geq \sum_{k = n}^{\infty}  \frac{1}{2^{t}q_{k}^{t}}  \sum_{j = 1}^{a_{k + 1}}  \frac{1}{j^{t}}\\
&\geq \frac{1 - 2^{t-1}}{2^{t}(1-t)} \sum_{i = n}^{\infty}  a_{m_{i}+1}^{1-t}  q_{m_{i}}^{-t}\\
&\geq \frac{1 - 2^{t-1}}{2^{t}(1-t)} \sum_{i = n}^{\infty}  \left( a_{m_{i}+1}  q_{m_{i}}^{1- 1/(1-t)} \right)^{1 - t}\\
&\geq \frac{1 - 2^{t-1}}{2^{t}(1-t)} \sum_{i = n}^{\infty}  \left( \min \left\{ 1, \frac{\textup{A}_{\alpha}(\theta)}{2} \right\} \right)^{1 - t} = \infty.
\end{align*}
The result follows from \eqref{eq:xyL}.

For the second part of the result, in \cite{KS:2012} it has already been shown that $d_{s, \delta}$ is a pseduo metric; and thus, it remains to show that $d_{s, \delta}(w, v) < \infty$, for all $w, v \in X$. However, this follows directly from \Cref{prop:curious,prop:curious-2,prop:lower_Holder_reg}.
\end{proof}

\begin{proof}[Proof of \Cref{prop:liminf_psi_r}]
The result follows from an application of \eqref{eq:unltrametric_shift-2}, \eqref{eq:xyL-2} and \Cref{lem:infinite_median_sum_t_1}, in tandem with the observation that 
\begin{align*}
\liminf_{n \to \infty} \phi(m, a_{m+2}, r, t) =  \liminf_{n \to \infty} (a_{m+2} q_{m+1} + q_{m})^{t(r - 1)} = \liminf_{n \to \infty} q_{m+2}^{t(r - 1)} = 0,
\end{align*}
where $\phi(m, a_{m+2}, r, t)$ is as defined in \eqref{eq:phi-1-2-3}.
\end{proof}

\subsection{Proof of \Cref{thm:thm2,thm:thm2.5}}\label{sec:proof_thm2+.5}

\begin{proof}[Proof of \Cref{thm:thm2}~\ref{thm:thm2(1)}]
Suppose that there is a $t \in (1- 1/\alpha, 1)$ so that the metric $d_{s, \delta}$ is sequentially $\overline{\varrho_{\alpha}(t)}$-H\"older regular to $d_{\delta}$, in which case $\psi_{X}(\rho_{\alpha}(t))$ is finite.  Let $t$ be fixed as such.  By this hypothesis we know that the metrics $d_{s, \delta}$ and $d_{\delta}$ are not Lipschitz equivalent and so by \Cref{rmk:alpha=1}, the continued fraction entries of $\theta$ are not bounded.  Let $a_{n_{m}}$ denote the $m$-th continued fraction entry of $\theta$, such that $a_{n_{m}+2} \geq 8$.  Since $(q_{m})_{m \in \mathbb{N}}$ is a monotonically increasing unbounded sequence, we notice
\begin{align*}
A_{\alpha}(\theta) = \limsup_{m \to \infty} a_{n_{m}} q_{n_{m}}^{1- \alpha}
\quad \text{and} \quad 
\left\lceil \frac{a_{n_{m}+2}}{2} \right\rceil q_{n_{m}+1} + q_{n_{m}} \leq \frac{2 (a_{n_{m}+2} q_{n_{m}+1} + q_{n_{m}})}{3}  = \frac{2q_{n_{m} +2}}{3}.
\end{align*}
Setting $r = \varrho_{\alpha}(t) = 1-(\alpha-1)/(\alpha t)$, we have that
\begin{align}\label{eq:A-theta_lower_bound_for_psi_r}
\begin{aligned}
\psi_{X}(r) &\geq \sup_{z\in\{x,y\}} \psi_{X,z}(r)\\
&\geq \limsup_{m\to\infty} \sup_{1 \leq j \leq a_{n_{m}+2}} \phi(n, j, r, t)\\
&\geq \limsup_{m\to\infty} \; ( \lceil a_{n_{m}+2}/2 \rceil q_{n_{m}+1} + q_{n_{m}})^{t r} 
\sum_{l = \lceil a_{n_{m}+2}/2 \rceil}^{a_{n_{m} + 2}} \frac{1}{(l q_{n_{m}+1} + q_{n_{m}})^{t}} \\
&\geq \limsup_{m \to \infty} \; ( \lceil a_{n_{m}+2}/2 \rceil q_{n_{m}+1} + q_{n_{m}})^{t r} \frac{q_{n_{m}+2}^{1-t} - ( \lceil a_{n_{m}+2}/2 \rceil q_{n_{m}+1} + q_{n_{m}})^{1 - t}}{(1-t) q_{n_{m} +1}}\\
&\geq \limsup_{m \to \infty}  \frac{1 - (2/3)^{1-t}}{2^{tr}(1-t)} \frac{q_{n_{m}+2}^{1-t(1-r)}}{q_{n_{m}+1}}\\
&= \frac{1 - (2/3)^{1-t}}{2^{tr}(1-t)} \limsup_{m \to \infty} \; \left(a_{n_{m}+2} q_{n_{m}+1}^{-t(1-r)/(1-t(1-r))}\right)^{1-t(1-r)}\\
&= \frac{1 - (2/3)^{1-t}}{2^{t - (\alpha - 1)/\alpha}(1-t)} \limsup_{m \to \infty} \; \left(a_{n_{m}+2} q_{n_{m}+1}^{1 - \alpha}\right)^{1/\alpha}\\
&= \frac{1 - (2/3)^{1-t}}{2^{t - (\alpha - 1)/\alpha}(1-t)} A_{\alpha}(\theta)^{1/\alpha}.
\end{aligned}
\end{align}
Hence, it follows that $A_{\alpha}(\theta)$ is finite.

The reverse implication is a consequence of \Cref{prop:curious,prop:curious-2,prop:lower_Holder_reg}.
\end{proof}

\begin{proof}[Proof of \Cref{thm:thm2}~\ref{thm:thm2(2)}]
For the forward implication, we prove the contra-positive; namely that, if $A_{\alpha}(\theta) = 0$, then $\psi_{X}(1-(\alpha-1)/(\alpha t)) = 0$. Using \Cref{lem:infinite_median_sum_t_1} and \Cref{prop:lower_Holder_reg} it is sufficient to show the following equalities. 
\begin{align*}
\limsup_{m \to \infty} q_{m}^{t (\alpha - (\alpha-1)/t)}  \sum_{k = m}^{\infty}  \sum_{j = 1}^{a_{k+1}} \frac{1}{(j  q_{k} + q_{k-1} - \mathds{1}_{2\mathbb{Z}}(k-m)q_{m-1})^{t}} &= 0\\
\limsup_{m \to \infty} \sup_{1 \leq j \leq a_{m+2}} (jq_{m+1} + q_{m})^{t (1 - (\alpha-1)/(\alpha t))}  \sum_{l = j}^{a_{m + 2}} \frac{1}{(l q_{m+1} + q_{m})^{t}} &= 0
\end{align*}
Using an identical argument to that presented in \eqref{eq:copy-paste-argument} yields the first equality; and using an identical argument to that presented in \eqref{eq:copy-paste-argument-2} yields the second equality.

The reverse implication, follows using an identical argument to that presented in \eqref{eq:A-theta_lower_bound_for_psi_r}.
\end{proof}

\begin{proof}[Proof of \Cref{thm:thm2.5}~\ref{thm:thm2(2)(c)}\ref{thm:thm2(2)(c)2}]
By the hypothesis we know that the metrics $d_{s, \delta}$ and $d_{\delta}$ are not Lipschitz equivalent and so by \Cref{rmk:alpha=1}, the continued fraction entries of $\theta$ are not bounded.  Let $a_{n_{m}}$ denote the $m$-th continued fraction entry of $\theta$, such that $a_{n_{m}+2} \geq 8$.  Thus since $(q_{m})_{m \in \mathbb{N}}$ is a monotonically increasing unbounded sequence, we have
\begin{align*}
A_{\alpha}(\theta) = \limsup_{m \to \infty} a_{n_{m}} q_{n_{m}}^{1- \alpha}
\quad \text{and} \quad
\left\lceil \frac{a_{n_{m}+2}}{2} \right\rceil q_{n_{m}+1} + q_{n_{m}} \leq \frac{2 (a_{n_{m}+2} q_{n_{m}+1} + q_{n_{m}})}{3}  = \frac{2q_{n_{m} +2}}{3}.
\end{align*}
Using \eqref{eq:unltrametric_shift-2}, \eqref{eq:xyL-2}, \Cref{lem:infinite_median_sum_t_1} and setting $j_{m} = \lceil a_{n_{m}+2}/2 \rceil$ and $r = \varrho_{\alpha}(t) = 1/\alpha$, we notice that
\begin{align*}
\psi_{X}(r) \geq \sup_{z\in\{x,y\}} \psi_{X,z}(r)
&\geq\limsup_{m \to \infty}  \phi(n_{m}, j_{m}, r, t)\\
&\geq \limsup_{m \to \infty}  ( \lceil a_{n_{m}+2}/2 \rceil q_{n_{m}+1} + q_{n_{m}})^{r} \frac{\ln(q_{n_{m+2}}) - \ln( \lceil a_{n_{m}+2}/2 \rceil q_{n_{m}+1} + q_{n_{m}})}{ q_{n_{m} +1}}\\
&\geq \limsup_{m \to \infty}  \frac{1}{2^{r}} \frac{q_{n_{m}+2}^{r}}{q_{n_{m}+1}} (\ln(q_{n_{m}+2}) - \ln(2 q_{n_{m}+2}/3))\\
&\geq \frac{\ln(3/2)}{2^{r}} \limsup_{m \to \infty}   a_{n_{m}+2}^{1/\alpha} q_{n_{m}+1}^{1/\alpha - 1}\\
&\geq \frac{\ln(3/2)}{2^{r}} \limsup_{m \to \infty}   (a_{n_{m}+2} q_{n_{m}+1}^{1 - \alpha})^{1/\alpha}\\
&\geq \frac{\ln(3/2)}{2^{r}} A_{\alpha}(\theta)^{1/\alpha}.
\end{align*}
This yields the required results.
\end{proof}

\begin{proof}[Proof of \Cref{thm:thm2.5}~\ref{thm:thm2(2)(c)}\ref{thm:thm2(2)(c)3}, \ref{thm:thm2(2)(c)}\ref{thm:thm2(2)(c)4}, \ref{thm:thm2(2)(a)}\ref{thm:thm2(2)(a)3} and \ref{thm:thm2(2)(a)}\ref{thm:thm2(2)(a)4}]
See \Cref{prop:curious,prop:curious-2,prop:lower_Holder_reg}.
\end{proof}

\begin{proof}[Proof of \Cref{thm:thm2.5}~\ref{thm:thm2(2)(b)}\ref{thm:thm2(2)(b)2}]
By the hypothesis we know that the metrics $d_{s, \delta}$ and $d_{\delta}$ are not Lipschitz equivalent and so by \Cref{rmk:alpha=1}, the continued fraction entries of $\theta$ are not bounded.  Let $a_{n_{m}}$ denote the $m$-th continued fraction entry of $\theta$, such that $a_{n_{m}+2} \geq 8$.  Thus since $(q_{m})_{m \in \mathbb{N}}$ is a monotonically increasing unbounded sequence, we have
\begin{align*}
A_{\alpha}(\theta) = \limsup_{m \to \infty} a_{n_{m}} q_{n_{m}}^{1- \alpha}
\quad \text{and} \quad
\left\lceil \frac{a_{n_{m}+2}}{2} \right\rceil q_{n_{m}+1} + q_{n_{m}} \leq \frac{2 (a_{n_{m}+2} q_{n_{m}+1} + q_{n_{m}})}{3}  = \frac{2q_{n_{m} +2}}{3}.
\end{align*}
Set $j_{m} = \lceil a_{n_{m}+2}/2 \rceil$ and $r = \varrho_{\alpha}(t)$, and let $\phi(m, j, r, t)$ be as in \eqref{eq:phi-1-2-3}. Using \eqref{eq:unltrametric_shift-2} and \eqref{eq:xyL-2} notice
\begin{align*}
\psi_{X}(r) &\geq
\limsup_{m \to \infty} \; \phi(n_{m}, j_{m}, r, t)\\
&\geq \limsup_{m \to \infty} \; ( \lceil a_{n_{m}+2}/2 \rceil q_{n_{m}+1} + q_{n_{m}})^{t r} \frac{( \lceil a_{n_{m}+2}/2 \rceil q_{n_{m}+1} + q_{n_{m}})^{1 - t} - q_{n_{m}+2}^{1-t}}{(t - 1) q_{n_{m} +1}}\\
&\geq \limsup_{m \to \infty}  \frac{(2/3)^{1-t} - 1}{2^{tr}(t-1)} \frac{q_{n_{m}+2}^{1-t(1-r)}}{q_{n_{m}+1}}\\
&\geq \frac{(2/3)^{1-t} - 1}{2^{tr}(t-1)} \limsup_{m \to \infty} a_{n_{m}+2}^{1-t(1-r)}q_{n_{m}+1}^{-t(1-r)}\\
&\geq \frac{(2/3)^{1-t} - 1}{2^{tr}(t-1)} \limsup_{m \to \infty} \left(a_{n_{m}+2} q_{n_{m}+1}^{-t(1-r)/(1-t(1-r))}\right)^{1-t(1-r)}\\
&\geq \frac{(2/3)^{1-t} - 1}{2^{t/\alpha}(t-1)} \limsup_{m \to \infty} \left(a_{n_{m}+2} q_{n_{m}+1}^{1-\alpha/(\alpha-t(\alpha-1))}\right)^{(\alpha-t(\alpha-1))/\alpha}\\
&\geq \frac{(2/3)^{1-t} - 1}{2^{t/\alpha}(t-1)} A_{(\alpha-t(\alpha-1))}(\theta)^{(\alpha-t(\alpha-1))/\alpha},
\end{align*}
where the last inequality holds since $t \in (1, \alpha/(\alpha-1))$ and hence $(\alpha-t(\alpha-1))/\alpha > 0$.  Thus, $A_{\alpha/(\alpha-t(\alpha-1))}(\theta) < \infty$.
\end{proof}

\begin{proof}[Proof of \Cref{thm:thm2.5}~\ref{thm:thm2(2)(b)}\ref{thm:thm2(2)(b)3}]
This is a consequence of \eqref{eq:spectral_metric_bound} and the fact that, for all $m \in \mathbb{N}$,
\begin{align*}
\sum_{n = m+1}^{\infty} \frac{1}{n^{t}}
\leq \frac{1}{t-1}  m^{-(t-1)}
=  \frac{1}{t-1}  (m^{-t})^{1-1/t}
\leq  \frac{1}{t-1}  (m^{-t})^{1/\alpha}.
\end{align*}
This completes the proof.
\end{proof}

\subsection{Proof of \Cref{thm:thm1}}\label{sec:Proof_Thm_1}

We divide the proof of \Cref{thm:thm1} into five parts: namely, we show the following implications:
\ref{thm:thm1(1)} $\Rightarrow$ \ref{thm:thm1(2)} $\Rightarrow$ \ref{thm:thm1(3)},
\ref{thm:thm1(3)} $\Rightarrow$ \ref{thm:thm1(1)},
\ref{thm:thm1(3)} $\Rightarrow$ \ref{thm:thm1(5)}, and
\ref{thm:thm1(5)} $\Rightarrow$ \ref{thm:thm1(2)}.

\begin{proof}[Proof of \Cref{thm:thm1}]
\ref{thm:thm1(1)} $\Rightarrow$ \ref{thm:thm1(2)}:
Assume that the statement is false, in which case either $\ell_{\alpha} = 0$ or $\ell_{\alpha} = \infty$.  First we consider the case $\ell_{\alpha} = 0$.  By definition of $\ell_{\alpha}$, there exist words $W, w \in \mathcal{L}(X)$ such that $w$ is a prefix and suffix of $W$, $W \neq w \neq \emptyset$ and
\begin{align}\label{eq:W-w-alpha}
1 \leq \lvert W \rvert - \lvert w \rvert \leq \left\lfloor \frac{\lvert w \rvert^{1/\alpha}}{2^{1/\alpha} R_{\alpha}^{1/\alpha}} \right\rfloor
\quad \text{and} \quad
R(n) \leq 2 R_{\alpha} n^{\alpha},
\end{align}
for all $n \geq \lvert w \rvert$.  Further, for all $i \in \{ 1, 2, \dots, \lvert w \rvert \}$, we have that 
\begin{align}\label{eq:W-w-i}
w_{i} = W_{i} = W_{i + \lvert W \rvert - \lvert w \rvert},
\end{align}
where we recall that $w_{k}$ and $W_{k}$ respectively denote the $k$-th letter of $w$ and $W$.  By the property of $\alpha$-repetitive, for all words $u \in \mathcal{L}(X)$ with
\begin{align*}
\lvert u \rvert = \left\lfloor \frac{\lvert w \rvert^{1/\alpha}}{2^{1/\alpha}R_{\alpha}^{1/\alpha}} \right\rfloor,
\end{align*}
we have that $u$ is a factor of $w$.  In particular, letting $\xi \in X$ and $k \in \mathbb{N}$, the factor 
\begin{align*}
\left(\xi_{k}, \xi_{k+1}, \dots, \xi_{k+ \lfloor \lvert w \rvert^{1/\alpha} 2^{-1/\alpha} R_{\alpha}^{-1/\alpha} \rfloor}\right),
\end{align*}
of $\xi$ is a factor of $w$.  This together with \eqref{eq:W-w-alpha} and \eqref{eq:W-w-i} yields that $\xi_{k} = \xi_{k + \lvert W \rvert - \lvert w \rvert}$ for all $k \in \mathbb{N}$, and thus, $\xi$ is periodic.  This contradicts the aperiodicity and minimality of $X$.  Therefore, if $X$ is $\alpha$-repetitive and not $\alpha$-repulsive, then $\ell_{\alpha} = \infty$.  For easy of notation set $B_{k} = \inf \{ A_{\alpha, n} \colon n \geq a_{k}  q_{k-1} \}$.  By \Cref{prop:blackboard}, for all $k \in \mathbb{N}$ we have that
\begin{align}\label{eq:Wwdefn}
W \coloneqq \underbrace{\mathcal{L}_{k} \dots \mathcal{L}_{k}}_{a_{2k}},
\quad
w \coloneqq \underbrace{\mathcal{L}_{k} \dots \mathcal{L}_{k}}_{a_{2k}-1},
\quad
W' \coloneqq \underbrace{\mathcal{R}_{k-1} \dots \mathcal{R}_{k-1}}_{a_{2k - 1}},
\quad \text{and} \quad
w' \coloneqq \underbrace{\mathcal{R}_{k-1} \dots \mathcal{R}_{k-1}}_{a_{2k - 1}-1}
\end{align}
all belong to the language $\mathcal{L}(X)$, that
\begin{align*}
\frac{\lvert W \rvert - \lvert w \rvert}{\lvert w \rvert^{1/\alpha}} = \frac{\lvert \mathcal{L}_{k} \rvert^{1-1/\alpha}}{(a_{2k} - 1)^{1/\alpha}} = \frac{q_{2k-1}^{1-1/\alpha}}{(a_{2k} - 1)^{1/\alpha}},
\end{align*}
provided that $a_{2k} \neq 1$, and that
\begin{align*}
\frac{\lvert W' \rvert - \lvert w' \rvert}{\lvert w' \rvert^{1/\alpha}} = \frac{\lvert \mathcal{R}_{k-1} \rvert^{1-1/\alpha}}{a_{2k-1} - 1} = \frac{q_{2(k-1)}^{1-1/\alpha}}{(a_{2k-1} - 1)^{1/\alpha}},
\end{align*}
provided that $a_{2k-1} \neq 1$. Hence, for $k \in \mathbb{N}$ with $a_{k} \neq 1$,
\begin{align}\label{eq:upper_bound_B_k}
B_{k} \leq q_{k-1}^{1-1/\alpha}(a_{k} - 1)^{-1/\alpha}.
\end{align}
Thus, since by assumption $\ell_{\alpha} = \infty$, since $B_{k} \leq B_{k +1}$, for all $k \in \mathbb{N}$, and since $(q_{k})_{k \in \mathbb{N}}$ is an unbounded monotonic sequence, given $N \in \mathbb{N}$ there exists $M \in \mathbb{N}$ so that $a_{j}  q_{j-1}^{1-\alpha} < N^{-\alpha}$, for all $j \geq M$.  For all $n \in \mathbb{N}$ let $m_{(n)}$ be the largest natural number so that $q_{m_{(n)}} \leq n$. By \Cref{thm:MH-1940}, for all $n \in \mathbb{N}$, so that $m_{(n)} \geq M$,
\begin{align*}
\frac{R(n)}{n^{\alpha}}
&\leq \frac{q_{m_{(n)}+1} + 2 q_{m_{(n)}} - 1 + q_{m_{(n)}+1} - q_{m_{(n)}}}{n^{\alpha}}\\
&\leq \frac{2 a_{m_{(n)}+1} q_{m_{(n)}} + 2q_{m_{(n)}-1} + q_{m_{(n)}}}{q_{m_{(n)}}^{\alpha}}
\leq \frac{2}{N^{\alpha}} + \frac{2q_{m_{(n)}-1}}{q_{m_{(n)}}^{\alpha}} + \frac{q_{m_{(n)}}}{q_{m_{(n)}}^{\alpha}}.
\end{align*}
Hence, we have that $R_{\alpha} \leq N^{-\alpha}$.  However, $N$ was chosen arbitrary and so $R_{\alpha} = 0$, this contradicts the initial assumption that $X$ is $\alpha$-repetitive.

\ref{thm:thm1(2)} $\Rightarrow$ \ref{thm:thm1(3)}:
Let $[0; a_{1} + 1, a_{2}, \dots]$ denote the continued fraction expansion of $\theta$.  Since the Sturmian subshift $X$ is $\alpha$-repulsive and $\alpha > 1$, by \Cref{prop:1-repulsive=repulsive} and \Cref{rmk:alpha-beta-repulsive,rmk:alpha=1}, we have that the continued fraction entries of $\theta$ are unbounded.  In particular, infinitely often we have that $a_{n} \neq 1$.  Setting $B_{k} = \inf \{ A_{\alpha, n} \colon n \geq a_{k}  q_{k-1} \}$, as in \eqref{eq:upper_bound_B_k}, we have that $B_{k} \leq q_{k-1}^{1-1/\alpha}(a_{k} - 1)^{-1/\alpha}$, for all $k \in \mathbb{N}$ with $a_{k} \neq 1$.  Since $B_{k} \leq B_{k+1}$, there exists $N \in \mathbb{N}$ so that, $2^{\alpha}/\ell_{\alpha}^{\alpha} \geq (a_{n} - 1)  q_{n-1}^{1-\alpha}$, for all $n \geq N$ with $a_{n} \neq 1$.  Hence, since the sequence $(q_{n})_{n \in \mathbb{N}}$ is an unbounded monotonic sequence and since, $X$ is $\alpha$-repulsive,
\begin{align*}
A_{\alpha}(\theta) = \limsup_{n \to \infty} a_{n} q_{n-1}^{1-\alpha} \leq \frac{2^{\alpha}}{\ell_{\alpha}^{\alpha}} < \infty.
\end{align*}
It remains is to show that $A_{\alpha}(\theta) > 0$.  We have observed that if the Sturmian subshift $X$ is $\alpha$-repulsive, then the continued fraction entries of $\theta$ are unbounded.  In particular, infinitely often we have that $a_{n} \neq 1$.  Thus, letting $W, w, W', w'$ be as in \eqref{eq:Wwdefn}, if $A_{\alpha}(\theta) = 0$, then $B_{k} = 0$, for all $k \in \mathbb{N}$, and hence that $\ell_{\alpha} = 0$.  This contradicts the assumption that $X$ is $\alpha$-repulsive.  Hence, if the Sturmian subshift $X$ is $\alpha$-repulsive, then $A_{\alpha}(\theta) > 0$.

\ref{thm:thm1(3)} $\Rightarrow$ \ref{thm:thm1(1)}:
Let $m_{(n)}$ denotes the largest integer so that $q_{m_{(n)}} < n$.  Since $A_{\alpha}(\theta) < \infty$, there exists a constant $c > 1$ so that $a_{m+1} \leq c q_{m}^{\alpha - 1}$, for all $m \in \mathbb{N}$.  By \Cref{thm:MH-1940} and the recursive definition of the sequence $(q_{n})_{n \in \mathbb{N}}$, we have that, for all $n \in \mathbb{N}$,
\begin{align*}
R(n)
&\leq R(q_{m_{(n)}}) + a_{m_{(n)}+1}q_{m_{(n)}}\\
&= 2 a_{m_{(n)}+1}q_{m_{(n)}} + q_{m_{(n)}-1} +2q_{m_{(n)}} - 1\\
&\leq 2c q_{m_{(n)}}^{\alpha} + q_{m_{(n)}-1} + 2 q_{m_{(n)}}\\
&\leq (2c + 3)n^{\alpha}\!\!.
\end{align*}
In particular, we have that, if $\theta$ is well-approximable of $\alpha$-type, then $R_{\alpha}$ is finite.  Further,  by \Cref{thm:MH-1940}, the recursive definition of the sequence $(q_{n})_{n \in \mathbb{N}}$ and the assumption that $A_{\alpha}(\theta) > 0$, we have that
\begin{align*}
R_{\alpha}
\geq \limsup_{k \in \mathbb{N}} \frac{R(q_{k})}{q_{k}^{\alpha}}
= \limsup_{k \in \mathbb{N}} \frac{q_{k+1} + 2q_{k} - 1}{q_{k}^{\alpha}}
\geq \limsup_{k \in \mathbb{N}} \frac{a_{k+1}q_{k}}{q_{k}^{\alpha}}
= A_{\alpha}(\theta) > 0.
\end{align*}
That is, if $\theta$ is well-approximable of $\alpha$-type, then $0 < R_{\alpha}$.

\ref{thm:thm1(3)} $\Rightarrow$ \ref{thm:thm1(5)}:
By \Cref{prop:blackboard} and the definition of $Q(n)$, we have $Q(q_{n}) \geq a_{n+1}$ and so
\begin{align*}
Q_{\alpha}
= \limsup_{n \to \infty} \frac{Q(n)}{n^{\alpha-1}}
\geq \limsup_{n \to \infty} \frac{Q(q_{n})}{q_{n}^{\alpha-1}}
\geq \limsup_{n \to \infty} \frac{a_{n+1}}{q_{n}^{\alpha-1}}
= A_{\alpha}(\theta) > 0.
\end{align*}
Thus, if $\theta$ is well-approximable of $\alpha$-type and $X$ was not $\alpha$-finite, then $Q_{\alpha}$ would be infinite.  By way of contradiction assume that $\theta$ is well-approximable of $\alpha$-type and $X$ and that $Q_{\alpha} = \infty$.  This means there exists a sequence of tuples $((n_{k}, p_{k}))_{k \in \mathbb{N}}$ of natural numbers such that
\begin{itemize}
\item the sequences $(n_{k})_{k \in \mathbb{N}}$ and $(p_{k})_{k \in \mathbb{N}}$ are strictly increasing and $\displaystyle{\lim_{n \to \infty} p_{k} n_{k}^{1 - \alpha} = \infty}$, and 
\item for each $k \in \mathbb{N}$ there exists a word $W_{(k)} \in \mathcal{L}(X)$ with $\lvert W_{(k)}\rvert = n_{k}$ and $\displaystyle{\underbrace{W_{(k)} W_{(k)} \cdots W_{(k)}}_{p_{k}} \in \mathcal{L}(X)}$.
\end{itemize}
For a fixed $k \in \mathbb{N}$, setting $\displaystyle W = \underbrace{W_{(k)} W_{(k)} \cdots W_{(k)}}_{p_{k}}$ and $\displaystyle w = \underbrace{W_{(k)} W_{(k)} \cdots W_{(k)}}_{p_{k}-1}$, we have that
\begin{align*}
\frac{\lvert W \rvert - \lvert w \rvert}{\lvert w \rvert^{1/\alpha}}
= \frac{n_{k}^{1 - 1/\alpha}}{(p_{k} - 1)^{1/\alpha}}
= \left(\frac{p_{k}}{p_{k} - 1} \frac{n_{k}^{\alpha-1}}{p_{k}} \right)^{1/\alpha}
= \left(\frac{p_{k}}{p_{k} - 1} \left(p_{k} n_{k}^{1- \alpha}\right)^{-1} \right)^{1/\alpha}\hspace{-1.5em}.
\end{align*}
This latter value converges to zero as $k$ increases to infinity.  Therefore, $\ell_{\alpha} = 0$ and so $X$ is not $\alpha$-repulsive.  This is a contradiction as have already seen that $\theta$ is well-approximable of $\alpha$-type if and only if $X$ is $\alpha$-repulsive.

\ref{thm:thm1(5)} $\Rightarrow$ \ref{thm:thm1(2)}:
Suppose that $Q_{\alpha}$ is non-zero and finite.  This means there is a sequence of tuples $((n_{k}, p_{k}))_{k \in \mathbb{N}}$ so that the sequence $(n_{k})_{k \in \mathbb{N}}$ is strictly monotonically increasing with $\displaystyle 0 < \lim_{k \to \infty} p_{k} n_{k}^{1 - \alpha} = Q_{\alpha} < \infty$, and for each $k \in \mathbb{N}$ there exists a word $W_{(k)} \in \mathcal{L}(X)$ with $\lvert W_{(k)}\rvert = n_{k}$ and 
\begin{align*}
\underbrace{W_{(k)} W_{(k)} \cdots W_{(k)}}_{p_{k}} \in \mathcal{L}(X).
\end{align*}
For a fixed $k \in \mathbb{N}$, setting $\displaystyle W = \underbrace{W_{(k)} W_{(k)} \cdots W_{(k)}}_{p_{k}}$ and $\displaystyle w = \underbrace{W_{(k)} W_{(k)} \cdots W_{(k)}}_{p_{k}-1}$, we have that
\begin{align*}
\frac{\lvert W \rvert - \lvert w \rvert}{\lvert w \rvert^{1/\alpha}}
= \frac{n_{k}^{1 - 1/\alpha}}{(p_{k} - 1)^{1/\alpha}}
= \left(\frac{p_{k}}{p_{k} - 1} \frac{n_{k}^{\alpha-1}}{p_{k}} \right)^{1/\alpha}\hspace{-1.5em}.
\end{align*}
This latter value converges to $Q_{\alpha}^{-1/\alpha}$, and so, we have that $\ell_{\alpha}$ is finite.

By way of contradiction, suppose $\ell_{\alpha} = 0$.  This implies there is a strictly increasing sequence of integers $(n_{m})_{m \in \mathbb{N}}$, so that there exist $W_{(n_{m})}, w_{(n_{m})} \in \mathcal{L}(X)$ with $W_{(n_{m})} \neq w_{(n_{m})}$, $\lvert W_{(n_{m})} \rvert = n_{m}$, $w_{(n_{m})}$ is a prefix and suffix of $W_{(n_{m})}$ and
\begin{align*}
\frac{\lvert W_{(n_{m})} \rvert - \lvert w_{(n_{m})} \rvert}{\lvert w_{(n_{m})} \rvert^{1/\alpha}} < \frac{1}{m}.
\end{align*} 
This means the two occurrences of $w_{(n_{m})}$ in $W_{(n_{m})}$ overlap.  Thus, there exist $p = p_{n_{m}} \in \mathbb{N}$ so that 
\begin{align*}
w = \underbrace{u \, u \, \cdots \, u}_{p-1} v \quad \text{and} \quad W = \underbrace{u \, u \, \cdots \, u}_{p} v,
\end{align*}
where $u = u_{(n_{m})}, v = v_{(n_{m})} \in \mathcal{L}(X)$ with $0 < \lvert v \rvert < \lvert u \rvert$.  Combing the above gives that $p  \lvert u \rvert^{1 - \alpha} > m^{\alpha}$, and so, $Q_{\alpha} = \infty$, contradicting the assumption that $Q_{\alpha}$ is finite.  
\end{proof}

\subsection{Proof of \Cref{thm:thm3}}\label{sec:proof_thm3}

\begin{proof}[Proof of \Cref{thm:thm3}]
For $\theta = [0; a_{1}, a_{2}, \dots ]$, it is known that,
\begin{align}\label{eq:Jarnik_bounds}
\frac{1}{(a_{n+1}+2) q_{n}^{2}}
\leq
\frac{1}{q_{n}(q_{n}+q_{n+1})}=
\left\lvert \frac{p_{n}+p_{n+1}}{q_{n}+q_{n+1}} - \frac{p_{n}  }{q_{n}  } \right\rvert
\leq
\left\lvert \theta - \frac{p_{n}}{q_{n}} \right\rvert
\leq \left\lvert \frac{p_{n}}{q_{n}} - \frac{p_{n+1}}{q_{n+1}} \right\rvert 
\leq \frac{1}{a_{n+1} q_{n}^{2}},
\end{align}
see for instance \cite{K:1997}.  Also, considering sequences of approximations $( p_{n}(x)/q_{n}(x) )_{n \in \mathbb{N}}$, of an irrational number $x = [0; a_{1} , a_{2}, \dots ] \in [0, 1] $, we have that
\begin{align*}
\mathcal{J}_{\alpha+1}^{1/c} \supseteq \left\{ x = [0; a_{1} , a_{2}, \dots ] \in [0, 1] \colon \left\lvert x - \frac{p_{n}(x)}{q_{n}(x)} \right\rvert \leq c^{-1}q_{n}(x)^{-\alpha - 1} \, \text{ for infinitely many } n\in\mathbb{N}  \right\},
\end{align*}
for further details see \cite{K:1997}.  Thus, by the lower bound in \eqref{eq:Jarnik_bounds}, if $\displaystyle \limsup_{n \to \infty} a_{n+1} q_{n}^{1-\alpha} \geq c$, for some given $c > 0$, then $\theta \in \mathcal{J}_{\alpha + 1}^{1/c}$.  Therefore,
\begin{align*}
\Theta_{\alpha} \subseteq \underline{\Theta}_{\alpha} \subseteq \{ \theta \in [0, 1] \colon A_{\alpha}(\theta) > 0 \} \subseteq \bigcup_{n \in \mathbb{N}} \mathcal{J}_{\alpha + 1}^{n},
\end{align*}
and so, by monotonicity and countable stability of the Hausdorff dimension (see for instance \cite{F:1990}) and \Cref{rm.:Jarnik_dimension}, we have that
\begin{align}\label{eq:Hausdorff_Dim_Lower_Final}
\textup{dim}_{\mathcal{H}}(\Theta_{\alpha}) \leq \textup{dim}_{\mathcal{H}}(\underline{\Theta}_{\alpha}) \leq 2/(\alpha + 1).
\end{align}
To prove that $2/(\alpha + 1)$ is a lower bound for $\textup{dim}_{\mathcal{H}}(\underline{\Theta}_{\alpha})$ and $\textup{dim}_{\mathcal{H}}(\Theta_{\alpha})$ we first show $\textup{Exact}(\alpha+1)$ is a subset of $\underline\Theta(\alpha)$ and $\overline\Theta(\alpha)$ and hence a subset of $\Theta(\alpha)$.  By \cite[Theorem 15]{K:1997} every best (reduced) rational approximation (of the first kind) $p/q$ to $\theta=[0;a_{1},a_{2},\ldots ]$, namely $|\theta-p'/q'|>|\theta-p/q|$, for all $p', q' \in \mathbb{N}$ with $q' < q$, is necessarily of the form $p^{(m)}/q^{(m)} = [0;a_{1},a_{2},\ldots, a_{n-1},m]$, for some $n\in \mathbb{N}$ and $1\leq m\leq a_{n}$.  In fact, $a_{n}/2 \leq m\leq a_{n}$, since if $m < a_{n}/2$, then by \eqref{eq:Jarnik_bounds},
\begin{align*}
\left\vert \theta - \frac{p^{(m)}}{q^{(m)}}\right\vert -  \left\vert \theta - \frac{p_{n-1}}{q_{n-1}}\right\vert
\geq \left\vert \frac{p_{n-1}}{q_{n-1}} - \frac{p^{(m)}}{q^{(m)}}\right\vert - 2 \left\vert \theta - \frac{p_{n-1}}{q_{n-1}}\right\vert
\geq \left\vert   \frac{p_{n-1}}{q_{n-1}} - \frac{p^{(m)}}{q^{(m)}}\right\vert - \frac{2}{q_{n}q_{n-1}}&\\[0.25em]
=  \frac{1}{q^{(m)}q_{n-1}} - \frac{2}{q_{n}q_{n-1}}
\geq \frac{(a_{n}-2m)q_{n-1}-q_{n-2}}{q^{(m)}q_{n}q_{n-1}} &> 0.
\end{align*}
Hence, $p^{(m)}/q^{(m)}$ is not a best approximation (of the first kind). From this, we conclude $1/2\leq q^{(m)}/q_{n}\leq 1$, for $a_{n}/2 \leq m\leq a_{n}$. Hence, for ever  reduced fraction $p/q$ with  $\vert \theta -p/q\vert\leq q^{-1-\alpha}$ we may assume without loss of generality that $p/q$ is a best approximation (of the first kind) and hence we find $n\in \mathbb{N}$ such that 
\begin{align*}
\left\vert \theta -\frac{p_{n}}{q_{n}}\right\vert \leq \left\vert \theta -\frac{p}{q}\right\vert\leq q^{-(\alpha + 1)} \leq 2^{\alpha + 1}q_{n}^{-(\alpha + 1)}.
\end{align*}
Using  the lower bound in \eqref{eq:Jarnik_bounds} gives, for every $\theta\in \textup{Exact}(\alpha+1)$, that $ \limsup  a_{n+1} q_{n}^{1-\alpha} \geq 2^{-(\alpha+1)}$ and thus that $\textup{Exact}(\alpha+1)\subset\underline\Theta(\alpha)$.  Further, assume that $ \left\vert \theta -{p}/{q}\right\vert > d q^{-(\alpha + 1)}$ for some $d<1$ and all but finitely many rationals $p/q$. This together with the upper bound in \eqref{eq:Jarnik_bounds} yields that $\limsup  a_{n+1} q_{n}^{1-\alpha} \leq d^{-1}$. In this way we have verified that $\textup{Exact}(\alpha+1)\subset\overline\Theta(\alpha)$.  The statement on the Hausdorff dimension of $\underline{\Theta}_{\alpha}$ and $\Theta_{\alpha}$ now follows from an application of \Cref{rm.:Jarnik_dimension}, the monotonicity of the Hausdorff dimension (see for instance \cite{F:1990}) and \eqref{eq:Hausdorff_Dim_Lower_Final}.

To complete the proof, we show that $\Lambda(\overline{\Theta}_{\alpha}) = 1$.  Notice, if $\theta \in [0, 1] \setminus \mathcal{J}_{\alpha + 1}^{1}$, using the upper bound given in \eqref{eq:Jarnik_bounds}, we have that $a_{n+1} q_{n}^{1 - \alpha} < 1$, for all but finitely many $n \in \mathbb{N}$, and thus, $A_{\alpha}(\theta) < 1$.  In particular, we have $\overline{\Theta}_{\alpha} \supseteq [0, 1] \setminus \mathcal{J}_{\alpha + 1}^{1}$.  This with \Cref{rm.:Jarnik_dimension} yields $\Lambda(\overline{\Theta}_{\alpha}) \geq \Lambda([0, 1] \setminus \mathcal{J}_{\alpha + 1}^{1}) = 1$.
\end{proof}

\section*{Acknowledgments}

This work was partly funded by the \emph{M8 Post-Doc-Initiative PLUS} program of the Universit\"at Bremen and the DFG Scientific Network \emph{Skew Product Dynamics and Multifractal Analysis} (OE 538/3-1).  The first author also acknowledges support from the DFG Emmy-Noether grant Ja 1721/2-1. The authors would like to thank Anna Zielicz for many motivating discussions concerning Sturmian subshifts and augmented trees.  We would also like to thank Johannes Kellendonk and Daniel Lenz for bringing our attention to the problems addressed in this article.

Revision of this articles were completed while M.\,Kesseb\"ohmer, A.\ Mosbach and T.\,Samuel were visiting the Mittag-Leffler institut as part of the program \textsl{Fractal Geometry and Dynamics}. We are extremely grateful to the organisers and staff for their very kind hospitality, financial support and a stimulating atmosphere.


\begin{thebibliography}{99}
%
\bibitem{A:2005}
B.\ Adamczewski.
On powers of words occurring in binary codings of rotations.
Adv. in Appl. Math. (1) 34 (2005), 1--29.
%
\bibitem{B:2014}
V.\ Berth\'e, V.\ Delecroix
Beyond substitutive dynamical systems: $S$-adic expansions.
RIMS Lecture note `Kokyuroku Bessatu' B46 (2014) 81--123.
%
\bibitem{B:2016}
V. Berth\'e,
$S$-adic expansions related to continued fractions.
RIMS Kyokuroku Bessatsu (2016).
%
\bibitem{STCS}
E.\ Christensen, C.\ Ivan.
Spectral triples for {AF $C^*$-algebras} and metrics on the {C}antor set.
J. Operator Theory (1) 56 (2006) 17--46.
%
\bibitem{expfast-finte-spec1}
E.\ Christensen, C.\ Ivan.
Sums of two dimensional spectral triples.
Math. Scand. (1) 100 (2007) 35--60.
%
\bibitem{DOST}
E.\ Christensen, C.\ Ivan, M.\ L.\ Lapidus.
Dirac operators and spectral triples for some fractal sets built on curves.
Adv. Math. (1) 217 (2008) 42--78.
%
\bibitem{AR:1991}
P.\ Arnoux, G.\ Rauzy.
Repr\'esentation g\'eom\'etrique de suites de complexit\'e $2n+1$.
Bull.\ Soc.\ Math.\ France (2) 119 (1991) 199--215. 
%
\bibitem{BG:2013}
M.\ Baake, U.\ Grimm.
Aperiodic order: A mathematical invitation. Vol.\ 1.
Encyclopedia of Mathematics and its Applications, 149, Cambridge Univ. Press, Cambridge, 2013.
%
\bibitem{BaMo:2000}
M.\ Baake, R.\ V.\ Moody (eds).
Directions in mathematical quasicrystals.
CRM Monogr.\ Ser., 13, Amer.\ Math.\ Soc., Providence, RI, 2000.
%
\bibitem{Bellisard}
J.\ V.\ Bellissard, J.\ Pearson.
Noncommutative Riemannian geometry and diffusion on ultrametric cantor sets.
J.\ Noncommut.\ Geom.\ (3) 3 (2009), 447--480.
%
\bibitem{BMR:2010}
J.\ V.\ Bellissard, M.\ Marcolli, K.\ Reihani.
{Dynamical Systems on Spectral Metric Spaces}.
Preprint, arXiv:1008.4617 (2010).
%
\bibitem{B:2012}
V.\ Beresnevich.
Rational points near manifolds and metric Diophantine approximation.
Ann.\ of Math. (1) 175 (2012), 187--235.
%
\bibitem{BBDV:2009}
V.\ Beresnevich, V.\ Bernik, M.\ Dodson and S.\ Velani.
Classical metric Diophantine approximation revisited.
W. W. L. Chen et al. (eds), Analytic Number Theory. Essays in Honour of Klaus Roth, 38--61, Cambridge Univ. Press, Cambridge, 2009.
%
\bibitem{BFMS:2002}
N.\ P.\ Fogg, V.\ Berth\'e, S.\ Ferenczi,\ C. Mauduit, A.\ Siegel (eds).
Substitutions in dynamics, arithmetics and combinatorics.
Lecture Notes in Mathematics, 1794. Springer-Verlag, Berlin, 2002. 
%
\bibitem{BGT:1987}
N.\ H.\ Bingham, C.\ M.\ Goldie, J.\ L.\ Teugels.
Regular variation.
Encyclopedia of Mathematics and its Applications, 27, Cambridge Univ. Press, Cambridge, 1989.
%
\bibitem{B:2003}
Y.~Bugeaud.
Sets of exact approximation order by rational numbers.
Ann.\ of Math. 327 (2003), 171--190.
%
\bibitem{B:2008}
Y.~Bugeaud.
Sets of exact approximation order by rational numbers II.
Unif.\ Distrib.\ Theory (2)  3 (2008), 9--20. 
%
\bibitem{NCDGCones}
A.\ Connes.
Noncommutative differential geometry.
Inst.\ Hautes {\'E}tudes Sci.\ Publ.\ Math.\ 62 (1985), 257--360.
 %
\bibitem{C:1989} 
A.\ Connes.
Compact metric spaces, Fredholm modules, and hyperfiniteness.
Ergodic Theory Dynam.\ Systems (2) 9 (1989), 207--220. 
%
\bibitem{C:1994}
A.\ Connes.
Noncommutative Geometry.
Academic Press, San Diego, CA, 1994.
%
\bibitem{DK:2002}
K.\ Dajani, C.\ Kraaikamp.
Ergodic theory of numbers.
Carus Mathematical Monographs, 29, Mathematical Association of America, Washington, DC, 2002.
%
\bibitem{DL:2002}
D.\ Damanik, D.\ Lenz.
The Index of Sturmian Sequences.
European J. Combin. (1) 23 (2002), 23--29.
%
\bibitem{LNCS4649}
V.\ Diekert, M.\ V.\ Volkov, A.\ Voronkov (ed).
Computer Science -- Theory and Applications.
Second International Symposium on Computer Science in Russia, CSR 2007.
Ekaterinburg, Russia, Sept. 2007, Proceedings.
%
\bibitem{DKMSS:2017}
F.\ Dreher, M.\ Kesseb\"ohmer, A.\ Mosbach, T.\ Samuel, M.\ Steffens.
Regularity of aperiodic minimal subshifts.
Bull.\ Math.\ Sci.\ (2017) 1--22.
%
\bibitem{D:2000}
F.\ Durand.
Linearly recurrent subshifts have a finite number of non-periodic subshift factors.
Ergodic Theory Dynam.\ Systems (4) 20 (2000), 1061--1078.
%
\bibitem{F:1990}
K.\ J.\ Falconer.
Mathematical foundations and applications. Third edition.
John Wiley \& Sons, Chichester, 2014.
%
\bibitem{Self-reference-1}
K.\ Falconer, T.\ Samuel.
Dixmier traces and coarse multifractal analysis.
Ergodic Theory Dynam.\ Systems (2) 31 (2011), 369--381. 
%
\bibitem{FGJ:2015}
G.\ Fuhrmann, M.\ Gr\"oger, T. J\"ager.
Amorphic complexity.
To appear in Nonlinearity.
%
\bibitem{Gelfand+Neumark}
I.\ M.\ Gelfand, M.\ A.\ Na\u{\i}mark.
On the imbedding of normed rings into the ring of operators in Hilbert space.
Rec. Math. [Mat. Sbornik] N.S. (54) 12 (1943), 197--213.
%
\bibitem{GLN:16}
R.\ Grigorchuk, D.\ Lenz, T.\ Nagnibeda.
Schreier graphs of Grigorchuk's group and a substitution associated to a non-primitive subshift.
To appear in:\ Groups, Graphs and Random Walks, T.~Ceccherini-Silberstein, M.~Salvatori, E.~Sava-Huss, (Eds), London Math.\ Soc.\ Lecture Note Series, Cambridge University Press (2017).
%
\bibitem{GLN:16b}
R.\ Grigorchuk, D.\ Lenz, T.\ Nagnibeda.
Spectra of Schreier graphs of Grigorchuk's group and Schroedinger operators with aperiodic order.
To appear in:\ Math.\ Ann.
%
\bibitem{GI1}
D.\ Guido, T.\ Isola.
Dimensions and singular traces for spectral triples, with applications to fractals.
J. Funct. Anal. (2) 203 (2003), 362--400.
%
\bibitem{GI2}
D.\ Guido, T.\ Isola.
Dimensions and spectral triples for fractals in $\mathbb{R}^{N}$.
Advances in Operator Algebras and Mathematical Physics, 89--108, Theta Ser.\ Adv.\ Math., 5, Theta, Bucharest, 2005.
%
\bibitem{HKW:2015}
A. Haynes, H. Koivusalo, J. Walton.
A characterization of linearly repetitive cut and project sets.
Preprint, arXiv:1503.04091 (2015).
%
\bibitem{INF:1985}
T.\ Ishimasa, H.\ U.\ Nissen, Y.\ Fukano.
New ordered state between crystalline and amorphous in Ni-Cr particles.
Phys.\ Rev.\ Lett.\ (5) 55 (1985), 511--513.
%
\bibitem{JP:2015}
A.\ Julien, I.\ F.\ Putnam.
Spectral triples for subshifts.
J.\ Funct.\ Anal.\ (3) 270 (2016),1031--1063.
%
\bibitem{K:2001}
V.\ A.\ Kaimanovich.
Random walks on Sierpi\'nski graphs: hyperbolicity and stochastic homogenization.
Fractals in Graz 2001, 145--183, Trends Math., Birh\"auser, Basel, 2003.
%
\bibitem{KLS:2011}
J.\ Kellendonk, D.\ Lenz, J.\ Savinien.
A characterization of subshifts with bounded powers
Discrete Math.\ (24) 313 (2013), 2881--2894.
%
\bibitem{KS:2012}
J.\ Kellendonk, J.\ Savinien.
Spectral triples and characterization of aperiodic order.
Proc.\ Lond.\ Math.\ Soc. (1) 104 (2012), 123--157.
%
\bibitem{KessSam:2013}
M.\ Kesseb\"ohmer, T.\ Samuel.
Spectral metric spaces for Gibbs measures.
J.\ Funct.\ Anal.\ (9) 265 (2013), 1801--1828. 
%
\bibitem{K:1997}
A.\ Ya.\ Khinchin.
Continued fractions.
Dover Publications, Mineola, NY, 1997.
%
\bibitem{Lapidus3}
M.\ L.\ Lapidus.
Towards a noncommutative fractal geometry? {L}aplacians and volume measures on fractals.
Harmonic analysis and nonlinear differential equations, 211--252, Contemp.\ Math., 208, Amer.\ Math.\ Soc., Providence, RI, 1997.
%
\bibitem{L:2008}
M.\ L.\ Lapidus.
In Search of the Riemann Zeros.
Amer.\ Math.\ Soc., Providence, RI, 2008
%
\bibitem{Lag:1999}
J.\ C.\ Lagarias.
Geometric models for quasicrystals I. Delone sets of finite type.
Discrete Comput.\ Geom.\ (29) 21 (1999), 161--191.
%
\bibitem{Lag:2002}
J.\ C.\ Lagarias, P.\ A.\ B.\ Pleasants.
Local Complexity of {De}lone sets and crystallinity.
Canad.\ Math.\ Bull.\ (4) 45 (2002), 634--652.
%
\bibitem{Lag:2003}
J.\ C.\ Lagarias, P.\ A.\ B.\ Pleasants.
Repetitive Delone sets and quasicrystals.
Ergodic Theory Dynam.\ Systems (3) 23 (2003), 831--867.
%
\bibitem{Lot:2002}
M.\ Lothaire.
Algebraic Combinatorics on Words.
Encyclopedia of Mathematics and its Applications, 90, Cambridge Univ. Press, Cambridge, 2002.
%
\bibitem{L:85}
I.\ G.\ Lysenok.
A set of defining relations for the Grigorchuk group.
(Russian), Mat. Zametki 38 (1985), 503--516.
English translation: Math.\ Notes 38 (1985), 784--792.
%
\bibitem{Mo:1997}
R.\ V.\ Moody (ed).
The mathematics of long-range aperiodic order.
Proceedings of the NATO Advanced Study Institute held in Waterloo, 403--441, NATO Adv.\ Sci.\ Inst.\ Ser.\ C Math.\ Phys.\ Sci., 489, Kluwer Academic Publishers Group, Dordrecht, 1997.
%
\bibitem{MH:1940}
G.\ A.\ Hedlund, M.\ Morse.
Symbolic dynamics II: Sturmian trajectories.
Amer.\ J.\ Math.\ 62 (1940), 1--42. 
%
\bibitem{P:1998}
J. Patera (ed).
Quasicrystals and discrete geometry.
Proceedings of the Fall Programme held at the University of Toronto. Fields Inst.\ Monogr., 10, Amer.\ Math.\ Soc., Providence, RI, 1998.
%
\bibitem{Pav}
B.\ Pavlovi\'c.
Defining metric spaces via operators from unital {$C^{*}$}-algebras.
 Pacific J. Math. (2) 186 (1998), 285--313.
%
\bibitem{Ri2}
M.\ A.\ Rieffel.
Metrics on states from actions of compact groups.
Doc. Math. 3 (1998), 215--229.
%
\bibitem{Ri3}
M.\ A.\ Rieffel.
Compact quantum metric spaces.
Operator algebras, quantization, and noncommutative geometry, 315--330, Contemp.\ Math., 365, Amer.\ Math.\ Soc., Providence, RI, 2004.
%
\bibitem{S:2015}
J.\ Savinien.
A metric characterisation of repulsive tilings.
Discrete Comput.\ Geom.\ (3) 54 (2015), 705--716. 
%
\bibitem{Sharp:2012}
R.\ Sharp.
Spectral triples and Gibbs measures for expanding maps on Cantor sets.
J.\ Noncommut.\ Geom.\ (4) 6 (2012), 801--817. 
%
\bibitem{SBGC:1984}
D.\ Shechtman, I.\ Blech, D.\ Gratias, J.\ W.\ Cahn.
Metallic phase with long-range orientational order and no translational symmetry.
Phys.\ Rev.\ Lett.\ (20) 53 (1984), 1951--1953.
%
\end{thebibliography}
\end{document}